\documentclass[a4paper]{amsart}
\usepackage{amscd}
\usepackage{amssymb}
\usepackage{algorithm}
\usepackage{algorithmic}
\usepackage[T1]{fontenc}
\usepackage[latin1]{inputenc}
\usepackage{hyperref}
\usepackage[arrow,curve,matrix]{xy}
\usepackage{times}
\usepackage{graphicx}
\usepackage{color}

\DeclareFontFamily{OMS}{rsfs}{\skewchar\font'60}
\DeclareFontShape{OMS}{rsfs}{m}{n}{<-5>rsfs5 <5-7>rsfs7 <7->rsfs10 }{}
\DeclareSymbolFont{rsfs}{OMS}{rsfs}{m}{n}
\DeclareSymbolFontAlphabet{\scr}{rsfs}

\renewcommand{\P}{\mathbb{P}} \newcommand{\C}{\mathbb{C}}
\renewcommand{\O}{\sO}
\newcommand{\Q}{\mathbb{Q}}

\newcommand\smallresto[1]{\hbox{\hbox{$_{{\vert}_{{#1}}}$}}}
\newcommand\resto[1]{\hbox{\hbox{$\big\vert{}_{_{#1}}$}}}
 
\newcommand{\sA}{\scr{A}}
\newcommand{\sB}{\scr{B}}
\newcommand{\sC}{\scr{C}}

\newcommand{\sF}{\scr{F}}

\newcommand{\sO}{\scr{O}}

\newcommand{\bA}{\mathbb{A}}

\newcommand{\bC}{\mathbb{C}}

\newcommand{\bN}{\mathbb{N}}

\newcommand{\bP}{\mathbb{P}}
\newcommand{\bQ}{\mathbb{Q}}

\newcommand{\bZ}{\mathbb{Z}}

\DeclareMathOperator{\codim}{codim}

\DeclareMathOperator{\Hilb}{Hilb}

\DeclareMathOperator{\red}{red}
\DeclareMathOperator{\reg}{reg}
\DeclareMathOperator{\sing}{sing}
\DeclareMathOperator{\Sym}{Sym}

\DeclareMathOperator{\Isom}{Isom}
\DeclareMathOperator{\Var}{Var}

\DeclareMathOperator{\Aut}{Aut}

\DeclareMathOperator{\length}{{length}}

\newcommand{\into}{\hookrightarrow}

\newcommand{\wt}{\widetilde}

\newcommand{\wtilde}{\widetilde}

\newcounter{thisthm}
\newcommand{\ilabel}[1]{\newcounter{#1}\setcounter{thisthm}{\value{thm}}\setcounter{#1}{\value{enumi}}}
\newcommand{\iref}[1]{(\thesection.\the\value{thisthm}.\the\value{#1})}

\theoremstyle{plain}    
\newtheorem{thm}{Theorem}[section]
\newtheorem{defn}[thm]{Definition}

\newtheorem{assumption}[thm]{Assumption} 
 
\numberwithin{equation}{thm}
\numberwithin{figure}{section}
\theoremstyle{plain}    
\newtheorem{cor}[thm]{Corollary}
\newtheorem{lem}[thm]{Lemma}

\newtheorem{fact}[thm]{Fact}

\theoremstyle{plain}    
\newtheorem{prop}[thm]{Proposition}
\newtheorem{proclaim-special}[thm]{\specialthmname}

\theoremstyle{remark}
\newtheorem{rem}[thm]{Remark}

\newtheorem{subrem}[equation]{Remark}
\newtheorem{subclaim}[equation]{Claim} 
\newtheorem*{claim*}{Claim} 
\newtheorem{notation}[thm]{Notation}
\newtheorem{example}[thm]{Example}

\newtheoremstyle{bozont-remark}{3pt}{3pt}%
     {}
     {}
     {\it}
     {.}
     {.5em}
     {\thmname{#1}\thmnumber{ #2}: \thmnote{\sc #3}}
\theoremstyle{bozont-remark}

\def\factor#1.#2.{\left. \raise 2pt\hbox{$#1$} \right/\hskip -2pt\raise
  -2pt\hbox{$#2$}} 

\newcommand\CounterStep{\addtocounter{thm}{1}}
\newlength{\swidth}
\newenvironment{enumerate-p}{
  \begin{enumerate}}
  {\setcounter{equation}{\value{enumi}}\end{enumerate}}

\date{\today}

\author{Stefan Kebekus}
\author{S\'andor J.\ Kov\'acs}

\thanks{Stefan Kebekus was supported in part by the DFG-Forschergruppe
  ``Classification of Algebraic Surfaces and Compact Complex
  Manifolds''.  S\'andor Kov\'acs was supported in part by NSF Grant
  DMS-0554697 and the Craig McKibben and Sarah Merner Endowed
  Professorship in Mathematics.}

\address{Stefan Kebekus, Mathematisches Institut, Universit\"at zu
  K\"oln, Weyertal 86--90, 50931 K\"oln, Germany}
\email{\href{mailto:stefan.kebekus@math.uni-koeln.de}{stefan.kebekus@math.uni-koeln.de}}
\urladdr{\href{http://www.mi.uni-koeln.de/~kebekus}{http://www.mi.uni-koeln.de/$\sim$kebekus}}

\address{\noindent S\'andor Kov\'acs, University of Washington,
  Department of Mathematics, Box 354350, Seattle, WA 98195, U.S.A.}
\email{\href{mailto:kovacs@math.washington.edu}{kovacs@math.washington.edu}}
\urladdr{\href{http://www.math.washington.edu/~kovacs}{http://www.math.washington.edu/$\sim$kovacs}}

\newcommand{\PreprintAndPublication}[2]{#1} 

\title[Surfaces mapping to the moduli stack of canonically polarized
varieties]{The structure of surfaces mapping to the moduli stack of
  canonically polarized varieties}

\begin{document}

\begin{abstract}
  Generalizing the well-known Shafarevich hyperbolicity conjecture, it
  has been conjectured by Viehweg that a quasi-projective manifold
  that admits a generically finite morphism to the moduli stack of
  canonically polarized varieties is necessarily of log general type.
  Given a quasi-projective surface that maps to the moduli stack, we
  employ extension properties of logarithmic pluri-forms to establish
  a strong relationship between the moduli map and the minimal model
  program of the surface. As a result, we can describe the fibration
  induced by the moduli map quite explicitly. A refined affirmative
  answer to Viehweg's conjecture for families over surfaces follows as
  a corollary.
\end{abstract}

\maketitle
\tableofcontents

\section{Introduction and main results}

\subsection{Introduction}

Let $S^\circ$ be a quasi-projective manifold that admits a morphism
$\mu: S^\circ \to \mathfrak M$ to the moduli stack of canonically
polarized varieties. Generalizing the classical Shafarevich
hyperbolicity conjecture, \cite{Shaf63}, Viehweg conjectured in
\cite[6.3]{Viehweg01} that $S^\circ$ is necessarily of log general
type if $\mu$ is generically finite. Equivalently, if $f^\circ:
X^\circ \to S^\circ$ is a smooth family of canonically pol\-arized
varieties, then $S^\circ$ is of log general type as soon as the
variation of $f^\circ$ is maximal, i.e., $\Var(f^\circ)=\dim S^\circ$.
We refer to \cite{KK05}, for the relevant notions, for detailed
references, and for a brief history of the problem.

Viehweg's conjecture was confirmed for $2$-dimensional manifolds
$S^\circ$ in \cite{KK05}; see also \cite{KS06}. Here, we complete the
picture. The cornerstone of the proof is an extension theorem for
logarithmic pluri-forms, Theorem~\ref{thm:logformextension}.  This
theorem and its consequences are used to establish a strong
relationship between the moduli map $\mu$ and the logarithmic minimal
model program of the surface $S^\circ$. This allows us to give a
complete description of the moduli map in those cases where the
variation cannot be maximal: the logarithmic minimal model program
always ends with a fiber space, and the family comes from the base of
this fibration, at least birationally and after suitable \'etale
cover.  Previous results and a refined affirmative answer to Viehweg's
conjecture for families over surfaces follow as a corollary.

The proof of our main result is rather conceptual and completely
independent of the argumentation of \cite{KK05} which essentially
relied on combinatorial arguments for curve arrangements on surfaces
and on Keel-McKernan's solution to the Miyanishi conjecture in
dimension 2\PreprintAndPublication{, \cite{KMcK}}{}.  The present
proof, besides giving a more complete picture, does not depend on the
Keel-McKernan result at all. Many of the techniques introduced here
generalize well to higher dimensions; most others at least
conjecturally.

\subsection{Main results}

The following is the main result of this paper.

\begin{thm}\label{thm:mainresult}
  Let $f^\circ: X^\circ \to S^\circ$ be a smooth projective family of
  canonically polarized varieties over a quasi-projective surface
  $S^\circ$ and $S$ a compactification of $S^\circ$ such that $D:=
  S\setminus S^\circ$ is a divisor with simple normal crossings.
  Assume that $\Var(f^\circ)>0$.
  
  Then $\kappa(S^\circ) \not = 0$. Furthermore, if $\kappa(S^\circ) <
  2$, then any log minimal model program of the pair $(S,D)$ will
  terminate at a fiber space, and the moduli map factors through the
  induced fibration of $S^\circ$. More precisely, we have the
  following:
  \begin{enumerate}
  \item If $\kappa({S^\circ})=-\infty$, then there exists an open set
    $U \subset S^\circ$ of the form $U = V \times \bA^1$ such that
    $X^\circ\resto U$ is the pull-back of a family over $V$. In
    particular, $\Var(f^\circ) = 1$.
  \item If $\kappa({S^\circ})=1$, then there exists an open set $U
    \subset S^\circ$ and a Cartesian diagram of one of the following
    two types, %
    \vskip-2.2em 
    $$
     \hskip1.3cm
    \xymatrix{%
      {\widetilde {U}} \ar[rr]^{\gamma}_{\text{\'etale}}
      \ar[d]_{\txt{\scriptsize $\widetilde \pi$\\\scriptsize elliptic
          \\ \scriptsize \quad fibration\quad\ }} && {{U}}
      \ar[d]^{\txt{\scriptsize $\pi$\vphantom{$\widetilde\pi$}\\
          \scriptsize elliptic \\
          \scriptsize \quad fibration\quad\ }} \\
      {\widetilde V} \ar[rr]_{\text{\'etale}} && {V}}
    \begin{minipage}[c]{.02\linewidth}
      \ \\
      \ \\
      \begin{center}
        \text{{or}}
      \end{center}
    \end{minipage}
    \xymatrix{%
      {\widetilde {U}} \ar[rr]^{\gamma}_{\text{\'etale}}
      \ar[d]_{\txt{\scriptsize
          $\widetilde \pi$\\ \scriptsize smooth\\
          \scriptsize algebraic\\ \scriptsize \quad
          $\bC^*$-bundle\quad\ }}
      && {{U}} \ar[d]^{\txt{\scriptsize $\pi\vphantom{\widetilde\pi}$\\
          \scriptsize smooth,  \\
          \scriptsize algebraic\\ \scriptsize \quad
          $\bC^*$-bundle\quad\ }}\\
      {\widetilde V} \ar@{=}[rr] && {V}}
    $$
    such that $X^\circ \times_U \tilde U$ is the pull-back of a family
    over $\widetilde V$. In particular, $\Var(f^\circ) = 1$.
  \end{enumerate}
\end{thm}

\begin{rem}
  Neither the compactification $S$ nor the minimal model program
  discussed in Theorem~\ref{thm:mainresult} is unique. When running
  the minimal model program, one often needs to choose the extremal
  ray that is to be contracted.
\end{rem}

A somewhat more precise version of Viehweg's conjecture for surfaces
also follows as an immediate consequence of
Theorem~\ref{thm:mainresult}\PreprintAndPublication{, cf.\
  \cite[Conjecture 1.6]{KK05}}{}.

\begin{cor}[\protect{Viehweg's conjecture for surfaces, \cite[Thm.~1.4]{KK05}}] 
  Let $f^\circ: X^\circ \to S^\circ$ be a smooth projective family of
  canonically polarized varieties over a quasi-projective surface
  $S^\circ$. Then either $\kappa(S^\circ)=-\infty$ and $\Var(f^\circ)
  < \dim S^\circ$, or $\Var(f^\circ) \leq \kappa(S^\circ)$. \qed
\end{cor}

\subsection{Conventions and notation}

Throughout the present paper we work over the complex number field.
When dealing with sheaves that are not necessarily locally free, we
frequently use square brackets to indicate taking the reflexive hull.

\begin{notation}
  Let $Y$ be a normal variety and $\sA$ a coherent sheaf of
  $\O_Y$-modules. Let $n\in \bN$ and set $\sA^{[n]} := (\sA^{\otimes
    n})^{**}$, $\Sym^{[n]} \sA := (\Sym^n \sA)^{**}$, etc. Likewise,
  for a morphism $f : X \to Y$ of normal varieties, set $f^{[*]}\sA :=
  (f^*\sA)^{**}$.
\end{notation}

We will later discuss the Kodaira dimension of singular pairs and the
Kodaira-Iitaka dimension of reflexive sheaves on normal spaces. Since
this is perhaps not quite standard, we recall the definition here.

\begin{notation}
  Let $Y$ be a normal projective variety and $\sA$ a reflexive sheaf
  of rank one on $Y$.  If $h^0\bigl(Y,\, \sA^{[n]}\bigr) = 0$ for all
  $n \in \mathbb N$, then we say that $\sA$ has Kodaira-Iitaka
  dimension $\kappa(\sA) := -\infty$.  Otherwise, recall that the
  restriction of $\sA$ to the smooth locus of $Y$ is locally free and
  consider the rational mapping
  $$
  \phi_n : Y \dasharrow \mathbb P\bigl(H^0\bigl(Y,\,
  \sA^{[n]}\bigr)^*\bigr).
  $$
  The Kodaira-Iitaka dimension of $\sA$ is then defined as
  $$
  \kappa(\sA) := \max_{n \in \mathbb N} \bigl(\dim
  \overline{\phi_n(Y)}\bigr).
  $$
  If $D \subset Y$ is an effective Weil divisor, define the Kodaira
  dimension of the pair $(Y,D)$ as $\kappa(Y,D) := \kappa
  \bigl(\sO_Y(K_Y+D) \bigr)$. If $Y$ is smooth and $D$ is a simple
  normal crossing divisor, define the Kodaira dimen\-sion of the
  complement $Y^\circ=Y\setminus D$ as $\kappa(Y^\circ)
  :=\kappa(Y,D)$.  Recall, that $\kappa(Y^\circ)$ is in\-dependent of
  the choice of the compactification $Y$.
\end{notation}

\subsection{Outline of proof, outline of paper}

The technical core of this paper is the extension result for pluri-log
forms, formulated in Theorems~\ref{thm:logformextension} and
\ref{thm:VZsheafextension2} of Section~\ref{sec:extendingForms}. In
essence, it states the following: If $(S,D)$ is a pair of a smooth
surface and a reduced divisor with simple normal crossings,
$(S_\lambda, D_\lambda)$ a log-minimal model, and $\sA_\lambda \subset
\Sym^{[n]} \Omega^1_{S_\lambda}(\log D_\lambda)$ any rank-one
reflexive sheaf of pluri-log forms, then $\sA_\lambda$ pulls back to a
reflexive sheaf of pluri-log forms in $\Sym^n \Omega^1_S(\log D')$,
where $D'$ is a divisor that is only slightly larger than $D$. Under
the conditions of Theorem~\ref{thm:mainresult}, a fundamental result
of Viehweg and Zuo asserts that a rank-one reflexive subsheaf
$\sA_\lambda\subset \Sym^{[n]} \Omega^1_{S_\lambda}(\log D_\lambda)$
of positive Kodaira-Iitaka dimension always exists.

The extension theorem is applied, e.g., in Section~\ref{sec:VZ}, in
order to give a criterion that is later used to show the fiber space
structure of certain minimal models. For an idea of the statement and
its proof, consider the setup of Theorem~\ref{thm:mainresult} in the
simplest case where $\kappa(S^\circ) = -\infty$. The log-minimal model
$(S_\lambda, D_\lambda)$ will then either be log-Fano of Picard-number
one, or a Mori-Fano fiber space. To show that $(S_\lambda, D_\lambda)$
is a Mori-Fano fiber space, we argue by contradiction and assume that
$\rho(S_\lambda) = 1$. Using this assumption and the existence of
$\sA_\lambda$, an analysis of the stability of
$\Omega^{[1]}_{S_\lambda}(\log D_\lambda)$ yields the existence of a
$\mathbb Q$-ample rank-one subsheaf $\sB_\lambda \subset
\Omega^{[1]}_{S_\lambda}(\log D_\lambda)$. The Extension Theorem will
then show the existence of a big invertible subsheaf $\sB \subset
\Omega^1_S(\log D')$. This, however, contradicts the well-known
Bogomolov-Sommese vanishing result, and the existence of a fiber space
structure is shown.

The argumentation in case $\kappa(S^\circ)=0$ follows a similar
outline, but is technically much more involved.
Section~\ref{sec:indexcoverk0} gathers results that are particular to
the case $\kappa=0$, work in any dimension and may be of independent
interest. The detailed description of the moduli map for fiber spaces
is done in a unified framework in Section~\ref{sec:gluearama}.

\subsection{Acknowledgments}

The work on which this article is based was finished while both
authors participated in the workshop ``Rational curves on algebraic
varieties'' at the American Institute of Mathematics in May, 2007. We
would like to thank the AIM for the stimulating atmosphere. We would
also like to thank J\'anos Koll\'ar for valuable suggestions that
undoubtedly made the article better, and Duco van Straten for a number
of discussions on the extension problem.

\part{TECHNIQUES}

\section{Extending pluri-forms over subvarieties of codimension one}
\label{sec:extendingForms}

If $X$ is a surface, $E \subset X$ a $(-1)$-curve and $\omega \in
H^0\bigl(X,\, \Omega^p_X(*E)\bigr)$ a $p$-form that is allowed to have
arbitrary poles along $E$, then an elementary computation shows that
$\omega$ is in fact everywhere regular on $X$, i.e., $\omega \in
H^0\bigl(X,\, \Omega^p_X\bigr)$.

Much of the argumentation in this paper is based on the observation
that a slightly weaker result also holds for pluri-log forms, and for
somewhat larger classes of divisors. \PreprintAndPublication{We refer
  to \cite{SS85, Flenner88} for more general extension results that
  apply to holomorphic $p$-forms.}{}

\subsection{Notation and standard facts about logarithmic differentials}

We introduce notation and recall two standard facts before stating and
proving the extension result in Section~\ref{sec:ext-ext} below. These
make sense and will be used both in the algebraic and in the analytic
category. We refer to \cite[Chapt.~11c]{Iitaka82} and
\cite[Chap.~3]{Deligne70} for details and proofs.

\begin{defn}\label{def:everythinglog}
  A \emph{reduced pair} $(Z,\Delta)$ consists of a normal variety $Z$
  and a reduced, but not necessarily irreducible Weil divisor $\Delta
  \subset Z$.  A \emph{morphism of reduced pairs} $\gamma: (\wtilde Z,
  \wtilde \Delta) \to (Z,\Delta)$ is a morphism $\gamma: \wtilde Z \to
  Z$ such that $\gamma^{-1}(\Delta) = \wtilde \Delta$
  set-theoretically.
  
  A reduced pair is called \emph{log smooth} if $Z$ is smooth and
  $\Delta$ has simple normal crossings. Given a reduced pair
  $(Z,\Delta)$, let $(Z,\Delta)_{\reg}$ be the maximal open set of $Z$
  where $(Z,\Delta)$ is log-smooth, and let $(Z,\Delta)_{\sing}$ be
  its complement, with the structure of a reduced subscheme. By a
  \emph{log-resolution} of $(Z, \Delta)$ we will mean a birational
  morphism of pairs $\gamma: (\wtilde Z, \wtilde \Delta) \to
  (Z,\Delta)$ where $(\wtilde Z, \wtilde \Delta)$ is log-smooth, and
  $\gamma$ is isomorphic along $(Z, \Delta)_{\reg}$.
\end{defn}

\begin{fact}[\protect{\cite{Hir62}}]
  Let $(Z,\Delta)$ be a reduced pair. Then a log-resolution exists.
  If $(Z,\Delta)$ is log-smooth, then the sheaf of log-differentials
  $\Omega^1_Z(\log \Delta)$ is locally free.  \qed
\end{fact}

\begin{fact}\label{fact:pullback}
  Let $\gamma: (\wtilde Z, \wtilde \Delta) \to (Z,\Delta)$ be a
  morphism of log-smooth reduced pairs, $U\subseteq Z$ an open set and
  $\wtilde U=\gamma^{-1}(U)$. Then there exists a natural pull-back
  map of log-forms
  $$
  \gamma^* :H^0\bigl(U,\, \Omega^1_Z(\log \Delta)\bigr) \to
  H^0\bigl(\wtilde U,\, \Omega^1_{\wtilde Z}(\log \wtilde \Delta)\bigr).
  $$
  and an associated sheaf morphism
  $$
  d\gamma: \gamma^*(\Omega^1_Z(\log \Delta)) \to \Omega^1_{\wtilde
    Z}(\log \wtilde \Delta).
  $$
  If $\gamma$ is finite and unramified over $Z \setminus \Delta$,
  then $d\gamma$ is isomorphic.  \qed
\end{fact}

\begin{subrem}
  The pull-back morphism also gives a pull-back of pluri-log forms,
  $$
  \gamma^* : H^0\bigl(Z,\,\Sym^n \Omega^1_Z(\log \Delta)\bigr) \to
  H^0\bigl(\wtilde Z,\, \Sym^n \Omega^1_{\wtilde Z}(\log \wtilde
  \Delta)\bigr),
  $$
  that obviously extends to a pull-back of rational forms.
\end{subrem}

We state one immediate consequence of Fact~\ref{fact:pullback} for
future reference.

\begin{cor}\label{cor:pb2}
  Under the conditions of Fact~\ref{fact:pullback}, assume that
  $\gamma$ is a finite morphism which is unramified over $Z \setminus
  \Delta$.  Let $E \subset Z$ be an effective divisor and $\sigma \in
  H^0\bigl(Z,\, \Sym^n\bigl(\Omega^1_Z(\log \Delta)\bigr)(*E)\bigr)$ a
  pluri-log form that might have poles along $E$.
  
  Then $\sigma$ has no poles along $E$, i.e., $\sigma \in
  H^0\bigl(Z,\, \Sym^n \Omega^1_Z(\log \Delta)\bigr)$ if and only if
  $\gamma^*(\sigma)$ has no poles along $\gamma^{-1}E$, i.e.,
  $\gamma^*(\sigma) \in H^0(\wtilde Z,\, \Sym^n \Omega^1_{\wtilde
    Z}(\log \wtilde \Delta))$.  \qed
\end{cor}
  
\begin{notation}
  In the setup of Corollary~\ref{cor:pb2}, we say that ``$\sigma$ has
  poles as a pluri-log form if and only if $\gamma^*(\sigma)$ has
  poles as a pluri-log form''.
\end{notation}

\subsection{Finitely dominated pairs}
\label{sec:ext-fdp}

The formulation of the main extension result in
Theorem~\ref{thm:logformextension} uses the following notion, which
slightly generalizes quotient singularities.

\begin{defn}
  A reduced pair $(Z,\Delta)$ is said to be \emph{finitely dominated
    by smooth analytic pairs} if for any point $z \in Z$, there exists
  an analytic neighborhood $U$ of $z$ and a finite, surjective
  morphism of reduced pairs $(\wtilde U, \wtilde \Delta) \to
  (U,\Delta\cap U)$ where $(\wtilde U, \wtilde\Delta)$ is log-smooth.
\end{defn}

Surface singularities that appear in certain variants of the minimal
model program are often finitely dominated by smooth analytic pairs.
In the rest of this subsection we discuss a class of examples that
will become important later.

\begin{defn}\label{def:dlc}
  A reduced pair $(Z,\Delta)$ is called \emph{dlc} if $(Z,\Delta)$ is
  {lc} and $Z\setminus \Delta$ is {lt}.
\end{defn}

\begin{example}
  It follows immediately from the definition that \emph{dlt} pairs are
  \emph{dlc}. For a less obvious example, let $Z$ be the cone over a
  conic and $\Delta$ the union of two rays through the vertex. Then
  $(Z,\Delta)$ is \emph{dlc}, but not \emph{dlt}.
\end{example}

\begin{lem}\label{ex:dltsurfisfd}
  Let $(Z,\Delta)$ be a {dlc} pair of dimension $2$.  Then
  $(Z,\Delta)$ is finitely dominated by smooth analytic pairs.  In
  particular, if $(Z,\Delta)$ is {dlt}, then it is finitely dominated
  by smooth analytic pairs.
\end{lem}
\begin{proof}
  Let $z \in (Z,\Delta)_{\sing}$ be an arbitrary singular point. If $z \not
  \in \Delta$, then the statement follows from \cite[4.18]{KM98}. We can
  thus assume without loss of generality for the remainder of the
  proof that $z \in \Delta$.

  To continue, observe that for any rational number $0<\varepsilon<1$,
  the non-reduced pair $\left(Z,(1-\varepsilon) \Delta\right)$ is
  \emph{numerically dlt}; see \cite[4.1]{KM98} for the definition and
  use \cite[3.41]{KM98} for an explicit discrepancy computation. By
  \cite[4.11]{KM98}, $Z$ is then $\bQ$-factorial. Using
  $\bQ$-factoriality, we can then choose a sufficiently small Zariski
  neighborhood $U$ of $z$ and consider the index-one cover for
  $\Delta\cap U$. This gives a finite morphism of pairs $\gamma :
  (\wtilde U, \wtilde \Delta) \to (U,\Delta\cap U)$, where the
  morphism $\gamma$ is branched only over the singularities of $U$,
  and where $\wtilde \Delta= \gamma^*(\Delta\cap U)$ is
  Cartier---see~\cite[5.19]{KM98} for the construction.  Choose any
  point $\wtilde z \in \gamma^{-1}(z)$. Since discrepancies only
  increase under taking finite covers, \cite[5.20]{KM98}, the pair
  $(\wtilde U, \wtilde \Delta)$ will again be {dlc}. In particular, it
  suffices to prove the claim for a neighborhood of $\wtilde z$ in
  $(\wtilde U, \wtilde \Delta)$. We can thus assume without loss of
  generality that $z \in \Delta$ and that $\Delta$ is Cartier in our
  original setup.

  Next, we claim that $(Z, \emptyset)$ is canonical at $z$.  In fact,
  let $E$ be any divisor centered above $z$, as in \cite[2.24]{KM98}.
  Since $z \in \Delta$, and since $\Delta$ is Cartier, the pull-back
  of $\Delta$ to any resolution where $E$ appears will contain $E$
  with multiplicity at least $1$. In particular, we have the following
  inequality for the log discrepancies: $0 \leq a(E, Z, \Delta) +1
  \leq a(E, Z, \emptyset)$.  This shows that $(Z, \emptyset)$ is
  canonical at $z$ as claimed.

  By \cite[4.20-21]{KM98}, $Z$ has a Du~Val quotient singularity at
  $z$. Again replacing $Z$ by a finite cover of a suitable
  neighborhood of $z$, and replacing $z$ by its preimage in the
  covering space, we can henceforth assume without loss of generality
  that $Z$ is smooth. But then the claim follows from
  \cite[4.15]{KM98}.
\end{proof}

\subsection{The extension theorem for finitely dominated pairs}
\label{sec:ext-ext}

The following is the main result of the present section. It asserts
that any log form defined outside of a divisor $E$ can be extended to
the whole space if $E$ contracts to a singularity which is finitely
dominated by a smooth analytic pair. Theorem~\ref{thm:logformextension} holds
in arbitrary dimension.

\begin{thm}[Extension Theorem for finitely dominated
  singularities]\label{thm:logformextension}
  Let $(Z,\Delta)$ be a reduced pair in the sense of
  Definition~\ref{def:everythinglog}, and assume that $(Z,\Delta)$ is
  finitely dominated by smooth analytic pairs.  Let $\psi: (Y,\Gamma)
  \to (Z,\Delta)$ be a log-resolution, $E_\Gamma \subset Y$ the union
  of the $\psi$-exceptional divisors that are not contained in
  $\Gamma$, and $n \in \bN$. Then $\psi_*\Sym^n \Omega^1_Y(\log
  (\Gamma + E_\Gamma))$ is reflexive.
\end{thm}
\begin{subrem}\label{rem:formextension}
  Let $E \subset Y$ be the exceptional set of $\psi$. Then
  Theorem~\ref{thm:logformextension} is equivalent to the statement
  that for any open set $U \subset Z$ with preimage $V :=
  \psi^{-1}(U)$ and any form $\sigma \in H^0 \bigl(V \setminus E,\,
  \Sym^n \Omega^1_{V \setminus E}(\log \Gamma) \bigr)$ defined outside
  the $\psi$-exceptional set $E \cap V$, the form $\sigma$ extends to
  a form $\wtilde \sigma \in H^0 \bigl(V,\, \Sym^n \Omega^1_V (\log
  (\Gamma + E_\Gamma)) \bigr)$ on all of $V$. Hence the name ``extension
  theorem''.
\end{subrem}
\begin{subrem}
  For an example in the simple case where $\Delta = \emptyset$, let $Y$ be
  the total space of $\sO_{\bP^1}(-2)$, and let $E$ be the
  zero-section. It is not very difficult to write down a pluri-log
  form
  $$
  \sigma \in H^0\bigl( Y,\, \Sym^2 \Omega^1_Y(\log E)\bigr) \setminus
  H^0\bigl( Y,\, \Sym^2 \Omega^1_Y\bigr).
  $$
  Because $E$ contracts to a quotient singularity, this example shows
  that Theorem~\ref{thm:logformextension} holds only for
  log-differentials, and that the boundary given in
  Theorem~\ref{thm:logformextension} is the smallest possible.

  \PreprintAndPublication{In order to construct $\sigma$, consider the
    standard coordinate cover of $Y$ with open sets $U_{1,2} \simeq
    \mathbb A^2$, where $U_i$ carries coordinates $x_i, y_i$ and
    coordinate change is given as
    $$
    \phi_{1,2} : (x_1, y_1) \mapsto (x_2, y_2) = ( x_1^{-1},\,
    x_1^2y_1).
    $$
    In these coordinates the bundle map $U_i \to \P^1$ is given as
    $(x_i, y_i) \to x_i$, and the zero-section $E$ is given as $E \cap
    U_i = \{y_i = 0\}$. Now take
    $$
    \sigma_2 := y_2^{-1}(dy_2)^2 \in \bigl( \Sym^2(\Omega^1_Y(\log
    E))\bigr)(U_2)
    $$
    and observe that $\phi_{1,2}^*(\sigma)$ extends to a form in
    $\bigl( \Sym^2(\Omega^1_Y(\log E))\bigr)(U_1)$.}{}
\end{subrem}

\begin{proof}[Proof of Theorem~\ref{thm:logformextension}]
  Assume that we are given an open set $U$ and a form $\sigma$ as in
  Remark~\ref{rem:formextension}. Since the extension problem is local
  on $Z$ in the analytic topology, we can shrink $Z$ and assume
  without loss of generality that there exists a finite, surjective
  morphism $\gamma : (\wtilde Z, \wtilde \Delta) \to (Z,\Delta)$ from a
  smooth pair $(\wtilde Z, \wtilde \Delta)$.
  
  Let $\wtilde Y$ be the normalization of $Y \times_{\wtilde Z} Z$ and
  $\wtilde \Gamma \subset \wtilde Y$ the reduced preimage of $\Gamma$.
  Then we obtain a commutative diagram of surjective morphisms of
  pairs as follows,
  $$
  \xymatrix{ (\wtilde Y, \wtilde \Gamma) \ar[rrr]^{\txt{\scriptsize
        $\wtilde \gamma$, finite}} \ar[d]_{\txt{\scriptsize $\wtilde
        \psi$\\\scriptsize contracts $\wtilde E$}} &&& (Y,\Gamma)
    \ar[d]^{\txt{\scriptsize$\psi$ \\\scriptsize
        log-resolution,\\\scriptsize contracts  $E$}} \\
    (\wtilde Z, \wtilde \Delta) \ar[rrr]_{\txt{\scriptsize $\gamma$,
        finite}} &&& (Z,\Delta) }
  $$
  where $\wtilde E := (\wtilde\gamma^{-1}(E))_{\red} =
  \bigl(\bigl(\gamma \circ \wtilde \psi\bigr)^{-1}
  (Z,\Delta)_{\sing}\bigr)_{\red}$ is the exceptional set of the
  morphism $\wtilde \psi$. Let $B \subset Z$ be the branch
  \emph{divisor} of $\gamma$, i.e., the minimal codimension-1 set such
  that $\gamma\resto {\wtilde Z \setminus \gamma^{-1}(B)}$ is étale in
  codimension one. Let $\psi^{-1}_*(B) \subset Y$ be its strict
  transform. Finally, set
  $$
  Y^0 := Y\setminus \bigl(\, \psi^{-1}_*(B) \, \cup\, \wtilde \gamma
  ((\wtilde Y, \wtilde \Gamma)_{\sing})\, \bigr).
  $$
  The set $Y^0$ is then the maximal open subset of $Y \setminus
  \psi^{-1}_*(B)$ such that $\wtilde Y^0 := \wtilde \gamma^{-1}(Y^0)$
  is contained in the log-smooth locus of $(\wtilde Y, \wtilde
  \Gamma)$.  We will use two of its main properties explicitly. These
  are contained in the following Claims.

  \begin{subclaim}\label{sc:1}
    The complement $Y\setminus Y^0$ intersects the $\psi$-exceptional
    set $E$ only in a set of codimension $\codim_Y(E \setminus Y^0)
    \geq 2$.
  \end{subclaim}
  \begin{proof}
    We need to show that
    \begin{align*}
      2 & \leq \codim_Y \big( \bigl(\, \psi^{-1}_*(B) \, \cup\,
      \wtilde \gamma
      ((\wtilde Y, \wtilde \Gamma)_{\sing})\, \bigr)\cap E \bigr) \\
      & = \min \{ \codim_Y (\psi^{-1}_*(B)\cap E), \, \codim_Y (\gamma
      ((\wtilde Y, \wtilde \Gamma)_{\sing})\,\cap E) \}.
    \end{align*}
    Since $\wtilde Y$ is normal, the log-singular locus $(\wtilde Y,
    \wtilde \Gamma)_{\sing}$ has codimension at least $2$. Since $\wtilde
    \gamma$ is finite, this gives $\codim_Y \wtilde \gamma
    \bigl((\wtilde Y, \wtilde \Gamma)_{\sing}\bigr) \geq 2$. It is also
    clear that $\psi^{-1}_*(B)$ and $E$ have no common component, so
    $\codim_Y \psi^{-1}_*(B)\cap E \geq 2$.
  \end{proof}
  
  \begin{subclaim}\label{sc:2}
    The morphism $\wtilde \gamma\resto {\wtilde Y^0}$ is étale outside of $E
    \cap Y^0$.
  \end{subclaim}
  \begin{proof}
    By construction, $\wtilde \gamma$ is étale outside of $E \cup
    \psi^{-1}_*(B)$.
  \end{proof}

  To prove Theorem~\ref{thm:logformextension}, we need to show that
  $\sigma$ extends to all of $Y$ as a pluri-log form, i.e., that the
  associated section
  $$
  \bar \sigma \in H^0 \bigl( Y,\, \bigl(\Sym^n \Omega^1_{Y}(\log
  (\Gamma + E_\Gamma)) \bigr)(*E) \big)
  $$
  has no poles along $E$ as a pluri-log form. Since $\bar \sigma$
  certainly has no poles outside of $E$, and since $\Sym^n
  \Omega^1_Y(\log (\Gamma + E_\Gamma))$ is locally free, Claim~\ref{sc:1}
  implies that it suffices to show that the restriction $\bar
  \sigma\resto {Y^0}$ has no poles along $E \cap Y^0$ as a pluri-log form.
  In particular, $\bar \sigma \in H^0 \bigl( Y,\, \Sym^n
  \Omega^1_{Y}(\log (\Gamma + E_\Gamma)) \bigr)$.
  
  By Corollary~\ref{cor:pb2} and Claim~\ref{sc:2}, it suffices to show
  that the pull-back $\wtilde \gamma^*(\bar \sigma)\resto {\wtilde
    Y^0}$ does not have any poles along $\wtilde Y^0 \cap \wtilde E$
  as a pluri-log form.  For that, recall that $\psi$ is an isomorphism
  over $(Z,\Delta)_{\reg}$. Hence the form $\sigma$ gives rise to a
  form
  $$
  \tau \in H^0 \bigl( Z,\, \Sym^{[n]} \Omega^1_Z(\log \Delta) \bigr).
  $$
  Since $(\wtilde Z, \wtilde \Delta)$ is log-smooth,
  Fact~\ref{fact:pullback} asserts that the pull-back of $\tau$
  extends to a pluri-log form $\wtilde \tau \in H^0 \bigl(\wtilde Z,\,
  \Sym^n \Omega^1_{\wtilde Z}(\log \wtilde \Delta) \bigr)$ on all of
  $\wtilde Z$. The pull-back to $\wtilde Y^0$,
  \begin{equation}\label{eq:htau}
    \hat \tau := {\wtilde\psi}^*(\wtilde \tau) \in
    H^0\bigl( \wtilde Y^0, \, \Sym^n \Omega^1_{\wtilde Y}(\log
    \wtilde \Gamma)\resto {\wtilde Y^0} \bigr),
  \end{equation}
  is then a pluri-log form on $\wtilde Y^0$ without poles that agrees
  with $\wtilde \gamma^*(\sigma)$ outside $\wtilde E$. This form
  necessarily equals $(\wtilde \gamma\resto {\wtilde Y^0})^*(\bar
  \sigma)$, which then does not have any poles along $\wtilde Y^0 \cap
  \wtilde E$, as asserted. Theorem~\ref{thm:logformextension} follows.
\end{proof}

\begin{cor}\label{cor:extension}
  Under the conditions of Theorem~\ref{thm:logformextension} we obtain
  an embedding
  $$
  \psi^{[*]} \Sym^{[n]} \Omega^1_Z(\log \Delta) \into \Sym^{n}
  \Omega^1_Y(\log (\Gamma + E_\Gamma)).
  $$
\end{cor}

\begin{proof}
  As $\psi$ induces an isomorphism $Y\setminus E\simeq Z\setminus
  \psi(E)$, Theorem~\ref{thm:logformextension} implies that
  $$
  \Sym^{[m]} \Omega^1_Z(\log \Delta)\simeq \psi_*\Sym^{m}
  \Omega^1_Y(\log (\Gamma + E_\Gamma))
  $$
  and hence we obtain that there exists a morphism
  $$
  \psi^*\Sym^{[m]} \Omega^1_Z(\log \Delta)\simeq \psi^*\psi_* \Sym^{m}
  \Omega^1_Y(\log (\Gamma + E_\Gamma)) \to \Sym^{m} \Omega^1_Y(\log
  (\Gamma + E_\Gamma)),
  $$
  which is an isomorphism, in particular an embedding, on $Y\setminus
  E$. This remains true after taking the double dual of these sheaves.
  Therefore the kernel of the map $\psi^{[*]} \Sym^{[m]}
  \Omega^1_Z(\log \Delta) \to \Sym^{m} \Omega^1_Y(\log (\Gamma +
  E_\Gamma))$ is a torsion sheaf and the fact that $\psi^{[*]}
  \Sym^{[m]} \Omega^1_Z(\log \Delta)$ is torsion-free implies the
  statement.
\end{proof}

\subsection{Extensions of Viehweg-Zuo sheaves}

We believe that the conclusion of Theorem~\ref{thm:logformextension}
holds for a larger class of singularities than those that we need to
discuss here. Thus it makes sense to introduce the following notation.

\begin{defn}\label{def:extthm}
  Let $(Z,\Delta)$ be a reduced pair in the sense of
  Definition~\ref{def:everythinglog}. Then we will say that \emph{the
    extension theorem holds for $(Z,\Delta)$} if for any
  log-resolution $\psi: (Y,\Gamma) \to (Z,\Delta)$, the sheaf
  $\psi_*\Sym^n \Omega^1_Y(\log (\Gamma + E_\Gamma))$ is reflexive,
  where $E_\Gamma$ denotes the union of the $\psi$-exceptional
  divisors that are not contained in $\Gamma$, and $n \in \bN$ is
  arbitrary.
\end{defn}

\begin{example}\label{ex:dltext}
  Example~\ref{ex:dltsurfisfd} and Theorem~\ref{thm:logformextension}
  imply that the extension theorem holds for {dlc} surface pairs.
\end{example} 

We will later consider log-smooth reduced pairs $(Z,\Delta)$ and
morphisms $f: Y \to Z$ whose restriction to $Z\setminus \Delta$ is a
smooth family of canonically polarized varieties. If $f$ has positive
variation, $\Var(f) > 0$, then Viehweg and Zuo have shown in
\cite[Thm.~1.4]{VZ02} that there exists a positive number $n$ and an
invertible subsheaf $\sA \subset \Sym^n \Omega^1_Z(\log \Delta)$ of
Kodaira-Iitaka dimension $\kappa(\sA) \geq \Var(f)$. We call this a
\emph{Viehweg-Zuo sheaf} on $(Z,\Delta)$.  More generally and more
precisely, we use the following definition.

\begin{defn}\label{def:VZ}
  Let $(Z,\Delta)$ be a reduced pair. A reflexive sheaf $\sA$ of rank
  $1$ is called a \emph{Viehweg-Zuo sheaf} if for some $n \in \mathbb
  N$ there exists an embedding $\sA \subset \Sym^{[n]}
  \Omega^1_Z(\log \Delta)$.
\end{defn}

The extension theorem will be used later to pull-back Viehweg-Zuo
sheaves to log resolutions. The following Theorem shows how this is
done.

\begin{thm}[Extension of Viehweg-Zuo sheaves]\label{thm:VZsheafextension2}
  Let $(Z,\Delta)$ be a reduced pair for which the extension theorem
  holds.  Using the setup of Definition~\ref{def:extthm}, assume that
  there exists a Viehweg-Zuo sheaf $\sA$ with inclusion $\iota : \sA
  \to \Sym^{[n]} \Omega^1_Z(\log \Delta)$. Then there exists an
  invertible Viehweg-Zuo sheaf $\sC \subset \Sym^n \Omega^1_Y(\log
  (\Gamma+E_\Gamma))$ with the following property: Let $m\in\bN$ and
  $$
  \iota^{[m]}: \sA^{[m]} \to \Sym^{[m\cdot n]} \Omega^1_Z(\log \Delta)
  $$
  the associated morphism of reflexive powers. Then $\iota^{[m]}$
  pulls back to give a sheaf morphism that factors through
  $\sC^{\otimes m}$,
  $$
  \bar \iota^{[m]} : \psi^{[*]}\sA^{[m]} \into \sC^{\otimes m}
  \subset \Sym^{m\cdot n} \Omega^1_Y(\log (\Gamma + E_\Gamma)).
  $$
\end{thm}
\begin{proof}
  By Corollary~\ref{cor:extension}, $\psi^{[*]}\sA$ embeds into
  $\Sym^{n} \Omega^1_Y(\log (\Gamma + E_\Gamma))$.  Let $\sC \subset
  \Sym^n \Omega^1_Y(\log (\Gamma + E_\Gamma))$ be the saturation of
  the image, which is automatically reflexive by \cite[Lem.~1.1.16 on
  p.~158]{OSS}. By \cite[Lem.~1.1.15 on p.~154]{OSS}, $\sC$ is then
  invertible as desired. Further observe that for any $m\in\bN$, the
  subsheaf $\sC^{\otimes m} \subset \Sym^{m\cdot n} \Omega^1_Y(\log
  (\Gamma + E_\Gamma))$ is likewise saturated.  Again, by
  Corollary~\ref{cor:extension}, there exists an embedding,
  $$
  \bar \iota^{[m]} : \psi^{[*]}\sA^{[m]} \into \Sym^{m\cdot n}
  \Omega^1_Y(\log (\Gamma + E_\Gamma)).
  $$
  It is easy to see that $\bar\iota^{[m]}$ factors through
  $\sC^{\otimes m}$ as it does so on the open set where $\psi$ is
  isomorphic, and because $\sC^{\otimes m}$ is saturated.
\end{proof}

\begin{rem}\label{rem:VZext2}
  Under the conditions of Theorem~\ref{thm:VZsheafextension2}, observe
  that the Kodaira-Iitaka dimension of $\sC$ is at least the
  Kodaira-Iitaka dimension of $\sA$, i.e., $\kappa(\sC) \geq
  \kappa(\sA)$.
\end{rem}

\section{Viehweg-Zuo sheaves on log minimal models}
\label{sec:VZ}

The existence of a Viehweg-Zuo sheaf of positive Kodaira-Iitaka
dimension clearly has consequences for the geometry of the underlying
space. The following theorem will later be used to show that a given
pair is a Mori-Fano fiber space. This will turn out to be a key step
in the proof of our main results. We refer to Definition~\ref{def:VZ}
for the notion of a Viehweg-Zuo sheaf.

\begin{thm}\label{thm:b2thm}
  Let $(Z,\Delta)$ be a reduced pair such that $Z$ is a normal and
  $\Q$-factorial surface. Assume that the following holds:
  \begin{enumerate}
  \item there exists a Viehweg-Zuo sheaf $\sA \subset \Sym^{[n]}
    \Omega^1_Z(\log \Delta)$ of positive Kodaira-Iitaka dimension,
  \item the extension theorem holds for $(Z,\Delta)$, and
  \item\ilabel{il:c} the anti-log-canonical divisor $-(K_Z + \Delta)$ is
    nef.
  \end{enumerate}
  Then $\rho(Z) > 1$.
\end{thm}

\begin{proof}
  We argue by contradiction and assume that $\rho(Z)=1$. Let $C
  \subset Z$ be a general complete intersection curve. Since $C$ is
  general, it avoids the singular locus $(Z,\Delta)_{\sing}$.
  By~\iref{il:c}, the restriction $\Omega^1_Z(\log \Delta)\resto {C}$ is a
  vector bundle of non-positive degree, \CounterStep
  \begin{equation}\label{eq:degneg}
    \deg \Omega^1_Z(\log \Delta)\resto C = (K_Z+\Delta).C \leq 0.
  \end{equation}
  
  We claim that the restriction $\Omega^1_Z(\log \Delta)\resto C$ is
  not anti-nef, i.e., that the dual vector bundle $\Omega^1_Z(\log
  \Delta)^*\resto C$ is not nef. In particular, we claim that
  $\Omega^1_Z(\log \Delta)\resto C$ admits a subsheaf of positive
  degree. Indeed, if $\Omega^1_Z(\log \Delta)\resto C$ were anti-nef,
  then none of its symmetric products $\Sym^n \Omega^1_Z(\log
  \Delta)\resto C$ could contain a subsheaf of positive degree.
  However, since $C$ is general, the restriction of the Viehweg-Zuo
  sheaf to $C$ is a locally free subsheaf $\sA\resto C \subset \Sym^n
  \Omega^1_Z(\log \Delta)\resto C$ of positive Kodaira-Iitaka
  dimension, and hence of positive degree. This proves the claim.
  
  As a consequence of the claim and of Equation~\eqref{eq:degneg}, we
  obtain that $\Omega^{[1]}_Z(\log \Delta)$ is not semi-stable and if
  $\sB \subset \Omega^{[1]}_Z(\log \Delta)$ denotes the maximal
  destabilizing subsheaf, its slope $\mu(\sB)$ is positive. The
  assumption that $\rho(Z)=1$ and $\mathbb Q$-factoriality then
  guarantees that $\sB$ is $\mathbb Q$-ample. In particular, its
  Kodaira-Iitaka dimension is maximal, $\kappa(\sB) = 2$.
  
  Now consider a log-resolution $\psi: (Y, \Gamma) \to (Z,\Delta)$ as
  in Definition~\ref{def:extthm}. The Extension Theorem for
  Viehweg-Zuo sheaves, Theorem~\ref{thm:VZsheafextension2},
  Remark~\ref{rem:VZext2}, and the assumption that $\rho(Z)=1$
  guarantee the existence of a Viehweg-Zuo sheaf $\sC \subset
  \Omega^1_Y(\log \Gamma+E_\Gamma)$ of Kodaira-Iitaka dimension
  $\kappa(\sC)=2$. As there are no symmetric tensors involved, this
  contradicts the Bogomolov-Sommese vanishing theorem,
  \cite[Cor.~6.9]{EV92}.
\end{proof}

\section{Global index-one covers for varieties of logarithmic Kodaira
  dimension 0}
\label{sec:indexcoverk0}

In this section, we consider a smooth pair $(Y,D)$ of Kodaira
dimension $0$, go to a minimal model and take the global index-one
cover.  If $(Y,D)$ carries a Viehweg-Zuo sheaf $\sA \subset \Sym^n
\Omega^1_Y(\log D)$ of positive Kodaira-Iitaka dimension, then we show
that the cover is uniruled and that its boundary is not empty. All
results of this section hold in arbitrary dimension.

\subsection{Construction of the cover}

First we briefly recall the main properties of the index-one cover, as
described in \cite[2.52]{KM98} or \cite[Sect.~3.6f]{Reid87}.

\begin{prop}\label{prop:index-cover}
  Let $(Y, D)$ be a reduced, log-smooth pair of dimension $\dim Y \geq
  2$ and Kodaira dimension $\kappa(Y,D) = 0$. Assume that there exists
  a birational map $\lambda : Y \dasharrow Y_\lambda$ to a normal
  variety $Y_\lambda$, such that the following holds. 
  \begin{enumerate-p}
  \item The inverse $\lambda^{-1}$ does not contract any divisor.
  \item $(Y_\lambda, D_\lambda)$ is a log minimal model of $(Y,D)$,
    where $D_\lambda$ denotes the cycle-theoretic image of $D$. 
  \item The log abundance conjecture holds for $(Y_\lambda,
    D_\lambda)$.
  \end{enumerate-p}
  Then there exists a diagram
  $$
  \xymatrix{ Y \ar@{-->}[d]_{\lambda} &&&& \wtilde Y
    \ar[d]^{\txt{\scriptsize $\wtilde \lambda$\\\scriptsize log
        resolution}} \\
    Y_\lambda &&&& \wtilde Y_\lambda \ar[llll]_{\txt{\scriptsize
        $\gamma$, index-one cover}}^{\txt{\scriptsize finite, étale
        where $Y_\lambda$ is smooth}} }
  $$
  with the following properties.
  \begin{enumerate-p}
  \setcounter{enumi}{\value{equation}}
\item\ilabel{il:431} If $\wtilde D_\lambda := \gamma^*(D_\lambda)$, then
  $K_{\wtilde Y_\lambda}+\wtilde D_\lambda$ is Cartier with
  $\O_{\wtilde Y_\lambda}(K_{\wtilde Y_\lambda}+\wtilde D_\lambda)
  \simeq \O_{\wtilde Y_\lambda}$
  \item\ilabel{il:433} The pair $(\wtilde Y_\lambda, \wtilde D_\lambda)$
    is {dlt}. If $y \in \wtilde Y_\lambda$ is a point where
    $(\wtilde Y_\lambda, \wtilde D_\lambda)$ is not log-smooth, then
    $(\wtilde Y_\lambda, \wtilde D_\lambda)$ is canonical at $y$.
  \item\ilabel{il:434} If $\wtilde D = \wtilde \lambda^*(\wtilde
    D_\lambda)_{\red}$, then $\kappa(\wtilde Y,\wtilde D)=0$.
  \end{enumerate-p}
\end{prop}

For the reader's convenience, we recall a few notions of higher
dimensional geometry used in the formulation of
Proposition~\ref{prop:index-cover}.

\begin{notation}
  A log minimal model is a {dlt} pair $(Y_{\lambda},
  D_{\lambda})$ where $Y_\lambda$ is $\bQ$-factorial and where
  $K_{Y_\lambda}+D_\lambda$ is nef, cf.~\cite[3.29--31]{KM98}.  If
  $(Y_{\lambda}, D_{\lambda})$ is a log minimal model and has Kodaira
  dimension $\kappa({Y_\lambda},D_\lambda)=0$, we say that \emph{the
    log abundance conjecture holds for $(Y_\lambda, D_\lambda)$} if
  there exists a number $k \in \bN^+$ such that $k \cdot
  (K_{Y_\lambda}+D_\lambda)$ is Cartier and $\O_{Y_\lambda}
  \bigl(k\cdot(K_{Y_\lambda}+D_\lambda)\bigr) \simeq \O_{Y_\lambda}$,
  cf.~\cite[3.12]{KM98}.
\end{notation}

\begin{rem}
  The existence of log minimal models and log abundance for minimal
  models is currently known for $\dim Y \leq 3$, see \cite[3.13]{KM98}
  for references concerning abundance. Both are expected to hold in
  any dimension\PreprintAndPublication{---see \cite{BCHM06, Siu06} for
    the latest progress}{}.
\end{rem}

\begin{proof}[Proof of Proposition~\ref{prop:index-cover}]
  Let $k \in \bN^+$ be the index of $K_{Y_\lambda}+D_\lambda$, i.e.,
  the smallest number such that
  $\O_{Y_\lambda}\bigl(k\cdot(K_{Y_\lambda}+D_\lambda)\bigr) \simeq
  \O_{Y_\lambda}$ and let $\gamma: \wtilde Y_\lambda\to Y_\lambda$ be
  the associated index-one cover.  We obtain that $K_{\wtilde
    Y_\lambda} + \wtilde D_\lambda$ is a Cartier divisor for the
  trivial bundle, as claimed in~\iref{il:431}.
  
  The assertion that $\wtilde Y_{\lambda}$ is {dlt} follows from the
  definition and from the fact that discrepancies increase under
  finite morphisms, \cite[5.20]{KM98}. If $y \in \wtilde Y_{\lambda}$
  is any point where $(\wtilde Y_\lambda, \wtilde D_\lambda)$ is not
  log-smooth, then by the definition of {dlt}, the discrepancy of any
  divisor $E$ that lies over $y$ is $a(E,\wtilde Y_{\lambda}, \wtilde
  D_{\lambda}) > -1$. But since $K_{\wtilde Y_{\lambda}} + \wtilde
  D_{\lambda}$ is Cartier, this number must be an integer, so
  $a(E,\wtilde Y_{\lambda}, \wtilde D_{\lambda}) \geq 0$.  It follows
  that the pair $(\wtilde Y_{\lambda}, \wtilde D_{\lambda})$ is
  canonical at $y$, hence \iref{il:433} is shown.
  
  It remains to prove that $\kappa({\wtilde Y}, \wtilde D)=0$, as
  claimed in~\iref{il:434}.  Since $(\wtilde Y_\lambda, \wtilde
  D_\lambda)$ is canonical wherever it is not log-smooth, the
  definition of canonical, \cite[2.26, 2.34]{KM98}, implies that
  $K_{\wtilde Y} +\wtilde D$ is represented by an effective, $\wtilde
  \lambda$-exceptional divisor, hence \iref{il:434} follows.
\end{proof}

\begin{cor}\label{cor:singofcoverdim2}
  Under the conditions of Proposition~\ref{prop:index-cover} further
  assume that $\dim Y = 2$. Then $(\wtilde Y_{\lambda}, \wtilde
  D_{\lambda})$ is log-smooth along $\wtilde D_{\lambda}$ and $\wtilde
  Y_\lambda$ is $\mathbb Q$-factorial.
\end{cor}
\begin{proof}
  The $\mathbb Q$-factoriality follows from~\iref{il:433} and
  \cite[4.11]{KM98}. Log-smoothness follows from the classification of
  canonical surface singularities, \cite[4.5]{KM98}.
\end{proof}

\subsection{The index-one cover in the presence of a Viehweg-Zuo
  sheaf}

We will later consider the index-one cover in the presence of a
Viehweg-Zuo sheaf $\sA$.  If $\kappa(\sA) > 0$, we will show that
$\wtilde Y$ is uniruled, and that the boundary cannot be empty. A
similar line of argumentation was used in \cite{KK05, KK07a}.

\begin{prop}\label{prop:unirulednessofbaseK0}
  Under the conditions of Proposition~\ref{prop:index-cover} further
  assume that there exists a Viehweg-Zuo sheaf $\sA \subset \Sym^n
  \Omega^1_Y(\log D)$ of positive Kodaira-Iitaka dimension,
  $\kappa(\sA) > 0$. Then $Y$ and $\wtilde Y$ are uniruled.
\end{prop}

The following---rather elementary---statements are used in the proof
of Proposition~\ref{prop:unirulednessofbaseK0}. We formulate two
separate lemmas for later reference.

\begin{lem}\label{lem:pushdownA}
  Let $(Y,D)$ be a log-smooth pair and assume that there exists a
  Viehweg-Zuo sheaf $\sA \subset \Sym^n \Omega^1_Y(\log D)$. If
  $\lambda : Y \dasharrow Y_\lambda$ is a birational map whose inverse
  does not contract any divisor, $Y_\lambda$ is normal and $D_\lambda$
  is the cycle-theoretic image of $D$, then there exists a Viehweg-Zuo
  sheaf $\sA_\lambda \subset \Sym^{[n]} \Omega^1_{Y_\lambda}(\log
  D_\lambda)$ of Kodaira-Iitaka dimension $\kappa(\sA_\lambda)\geq
  \kappa(\sA)$.
\end{lem}
\begin{proof}
  The assumption that $\lambda^{-1}$ does not contract any divisors
  and the normality of $Y_\lambda$ guarantee that $\lambda^{-1}:
  Y_\lambda \dasharrow Y$ is well-defined and injective over an open
  subset $Y_\lambda^\circ \subset Y_\lambda$ whose complement has
  codimension $\codim_{Y_\lambda} (Y_\lambda \setminus
  Y_\lambda^\circ) \geq 2$.  In particular,
  $D_\lambda\resto{Y_\lambda^\circ} = \bigl(
  \lambda^{-1}\resto{Y_\lambda^\circ} \bigr)^{-1}D$. Let $\iota:
  Y_\lambda^\circ \into Y_\lambda$ denote the embedding and set
  $\sA_\lambda := \iota_* \bigl(
  (\lambda^{-1}\resto{Y_\lambda^\circ})^{[*]}\sA \bigr)$.
  Fact~\ref{fact:pullback} gives an inclusion $\sA_\lambda \subset
  \Sym^{[n]}\Omega^1_{Y_\lambda}(\log D_\lambda)$. By construction
  $h^0\bigl(Y_\lambda,\, \sA_\lambda^{[m]}\bigr) \geq h^0(Y,\,
  \sA^{\otimes m})$ for all $m>0$, hence $\kappa(\sA_\lambda)\geq
  \kappa(\sA)$.
\end{proof}

\begin{lem}\label{lem:VZontilde}
  Under the conditions of Proposition~\ref{prop:index-cover} further
  assume that there exists a Viehweg-Zuo sheaf $\sA \subset \Sym^n
  \Omega^1_Y(\log D)$. Then there exists a Viehweg-Zuo sheaf $\wt
  \sA_\lambda \subset \Sym^{[n]} \Omega^1_{\wtilde Y_\lambda}(\log
  \wtilde D_\lambda)$ of Kodaira-Iitaka dimension $\kappa(\wt
  \sA_\lambda) \geq \kappa(\sA)$.
\end{lem}
\begin{proof}
  Let $\sA_\lambda$ be defined as in Lemma~\ref{lem:pushdownA}, and
  set $\wt \sA_\lambda := \gamma^{[*]}\sA_\lambda$. The facts that
  $\wt \sA_\lambda$ is reflexive and that $\gamma$ is étale imply that
  there exists an embedding $\wt \sA_\lambda \to \Sym^{[n]}
  \Omega^1_{\wtilde Y_\lambda}(\log \wtilde D_\lambda)$, as claimed.
\end{proof}

\begin{proof}[Proof of Proposition~\ref{prop:unirulednessofbaseK0}]
  Since uniruledness is a birational property, and since images of
  uniruled varieties are again uniruled, it suffices to show the claim
  for $\wtilde Y_\lambda$.  We argue by contradiction and assume that
  $\wtilde Y_\lambda$ (and then also $\wtilde Y$) is \emph{not} uniruled
  ---by \cite[Cor.~0.3]{BDPP03} this is equivalent to assuming that
  $K_{\wtilde Y}$ is pseudo-effective.  Again by
  \cite[Thm.~0.2]{BDPP03}, this is in turn equivalent to the
  assumption that $K_{\wtilde Y}\cdot C \geq 0$ for all moving curves
  $C \subset \wtilde Y$.
  
  As a first step, we will show that the assumption implies that the
  (Weil) divisor $\wtilde D_\lambda$ is zero. To this end, choose a
  polarization of $\wtilde Y_{\lambda}$ and consider a general
  complete intersection curve $\wtilde C_{\lambda} \subset \wtilde
  Y_{\lambda}$.  Because $\wtilde C_{\lambda}$ is a complete
  intersection curve, it intersects the support of the effective
  divisor $\wtilde D_\lambda$ if the support is not empty.  By general
  choice, the curve $\wtilde C_{\lambda}$ is contained in the smooth
  locus of $\wtilde Y_\lambda$ and avoids the indeterminacy locus of
  $\wtilde \lambda^{-1}$.  Its preimage $\wtilde C := \wtilde
  \lambda^{-1}(\wtilde C_\lambda)$ is then a moving curve in $\wtilde
  Y$ which intersects $\wtilde D$ positively if and only if the Weil
  divisor $\wtilde D_\lambda$ is not zero.  But
  $$
  0 = (K_{\wtilde Y_\lambda} + \wtilde D_\lambda)\cdot \wtilde
  C_\lambda = (K_{\wtilde Y} + \wtilde D)\cdot \wtilde C =
  \underbrace{K_{\wtilde Y} \cdot \wtilde C}_{\geq 0, \text{ as
      $\wtilde C$ is moving}} + \underbrace{\wtilde
    D\vphantom{_{\wtilde Y}} \cdot \wtilde C}_{\geq 0, \text{ as
      $\wtilde C$ not in $\wtilde D$}},
  $$
  so $\wtilde D \cdot \wtilde C = 0$. In particular, $\wtilde
  D_\lambda$ is the zero divisor. This, combined with the fact that
  $\O_{\wtilde Y_\lambda}(K_{\wtilde Y_\lambda}+\wtilde D_\lambda)
  \simeq \O_{\wtilde Y_\lambda}$ implies that the canonical divisor
  $K_{\wtilde Y_\lambda}$ is Cartier and its associated sheaf is
  trivial. In particular, the restrictions $\Omega^1_{\wtilde
    Y_\lambda}\resto{\wtilde C_\lambda}$ and $T_{\wtilde
    Y_\lambda}\resto{\wtilde C_\lambda}$ are vector bundles of degree zero
  and so is the symmetric product $\Sym^n \Omega^1_{\wtilde
    Y_\lambda}\resto{\wtilde C_\lambda}$.
  
  Recall from Lemma~\ref{lem:VZontilde} that there exists a
  Viehweg-Zuo sheaf of positive Kodaira-Iitaka dimension, say $\wt
  \sA_\lambda \subset \Sym^{[n]} \Omega^1_{\wtilde Y_\lambda}(\log
  \wtilde D_\lambda)$. As ${\wtilde C_\lambda}$ is a general curve,
  the restriction $\wt \sA_\lambda\resto{\wtilde C_\lambda} \subset
  \Sym^n \Omega^1_{\wtilde Y_\lambda}\resto{\wtilde C_\lambda}$ has
  positive degree. In particular, $\Sym^n \Omega^1_{\wtilde
    Y_\lambda}\resto{\wtilde C_\lambda}$ is not semi-stable. Since
  symmetric products of semi-stable vector bundles are again
  semi-stable \cite[Cor.~3.2.10]{HL97}, this implies that
  $\Omega^1_{\wtilde Y_\lambda}\resto{\wtilde C_\lambda}$ is likewise
  not semi-stable. Since $\deg\Omega^1_{\wtilde
    Y_\lambda}\resto{\wtilde C_\lambda}=\deg T_{\wtilde
    Y_\lambda}\resto{\wtilde C_\lambda}=0$, this also implies that
  $T_{\wtilde Y_\lambda}\resto{\wtilde C_\lambda}$ is not semi-stable.

  In particular, the maximal destabilizing subsheaf of $T_{\wtilde
    Y_\lambda}\resto{\wtilde C_\lambda}$ is of positive degree, hence
  ample. In this setup, a variant \cite[Cor.~5]{KST07} of Miyaoka's
  uniruledness criterion \cite[Cor.~8.6]{Miy85} applies to give the
  uniruledness of $\wtilde Y_\lambda$. For more details on this
  criterion see the survey \cite{KS06}. This ends the proof of
  Proposition~\ref{prop:unirulednessofbaseK0}.
\end{proof}

\begin{cor}\label{cor:boundarynonempty}
  Under the conditions of Proposition~\ref{prop:unirulednessofbaseK0}
  the boundary divisor $D_\lambda$ is not empty. In particular, $D$,
  $\wtilde D_\lambda$ and $\wtilde D$ are not empty.
\end{cor}
\begin{proof}
  Again, we assume to the contrary that $D_\lambda$ is empty.
  Proposition~\ref{prop:index-cover} then implies that
  $\kappa({\wtilde Y})=0$, while
  Proposition~\ref{prop:unirulednessofbaseK0} asserts that $\wtilde Y$
  is uniruled, a contradiction.
\end{proof}

\section{Unwinding families}
\label{sec:gluearama}

We will consider projective families $g: Y \to T$ where the
base $T$ itself admits a fibration $\rho: T \to B$ such that $g$ is
isotrivial on all $\rho$-fibers. It is of course generally false that
$g$ would be the pull-back of a family defined over $B$. We will,
however, show in this section that in some situations the family $g$
does become a pull-back after a suitable base change.

We use the following notation for fibered products that appear in our
setup.

\begin{notation}
  Let $T$ be a scheme, $Y$ and $Z$ schemes over $T$ and $h:Y\to Z$ a
  $T$-morphism. If $t\in T$ is any point, let $Y_t$ and $Z_t$ denote
  the fibers of $Y$ and $Z$ over $t$. Furthermore, let $h_t$ denote
  the restriction of $h$ to $Y_t$. More generally, for any $T$-scheme
  $\wtilde T$, let
  $$
  h_{\wtilde T}: \underbrace{Y\times_T {\wtilde T}}_{=: Y_{\wtilde T}}
  \to \underbrace{Z\times_T {\wtilde T}}_{=: Z_{\wtilde T}}
  $$
  denote the pull-back of $h$ to $\wtilde T$. The situation is
  summarized in the following commutative diagram.
  $$
  \xymatrix{
    Y_{\wtilde T} \ar[dr] \ar@/^1.25pc/[rrr] \ar[rr]_{h_{\wtilde T}} & &
    Z_{\wtilde T} \ar[dl] \ar@/^1.25pc/[rrr]|(.17)\hole & Y
    \ar[dr] \ar[rr]_{h}  && Z \ar[dl] \\
     & \wtilde T \ar[rrr] & & & T }
  $$
\end{notation}

The setup of the current section is then formulated as follows.

\begin{assumption}\label{ass:triplemor}
  Throughout the present section, consider a sequence of morphisms
  between algebraic varieties,
  $$
  \xymatrix{ Y \ar[rr]^{g}_{\text{smooth, projective}} && {\ T\ }
    \ar[rr]^{\rho}_{\text{smooth, rel.~dim.=1}} && B, }
  $$
  where $g$ is a smooth projective family and $\rho$ is smooth of
  relative dimension 1, but not necessarily projective.  Assume
  further that for all $b \in B$, there exists a smooth variety $F_b$
  such that for all $t \in T_b$, there exists an isomorphism $Y_t
  \simeq F_b$.
\end{assumption}

\subsection{Relative isomorphisms of families over the same base}

To start, recall the well-known fact that an isotrivial family of
varieties of general type over a curve becomes trivial after passing
to an étale cover of the base. As we are not aware of an adequate
reference, we include a proof here.

\begin{lem}\label{lem:dominant}
  Let $b\in B$ and assume that $\Aut(F_b)$ is finite.  Then the
  natural morphism $\iota: I=\Isom_{T_b}({Y_b}, {T_b}\times F_b)\to
  {T_b}$ is finite and étale. Furthermore, pull-back to $I$ yields an
  isomorphism of $I$-schemes $Y_I \simeq I\times F_b$.
\end{lem}

\begin{proof}
  Consider the ${T_b}$-scheme
  $$
  H := \Hilb_{T_b} \bigl( {Y_b} \times_{T_b} ({T_b} \times F_b) \bigr) \simeq
  \Hilb_{T_b} \bigl( {Y_b} \times F_b \bigr).
  $$
  By Assumption~\ref{ass:triplemor}, $H_t\simeq \Hilb(F_b\times F_b)$
  for all $t\in {T_b}$.  Similarly, $I_t\simeq \Aut(F_b)$, hence $I$
  is one-dimensional and $\length(I_t)$ is constant on ${T_b}$.  Since
  $I$ is open in $H$, the union of components of $H$ that contain $I$,
  denoted by $H^I$, is also one-dimensional.
  
  Recall that $H\to {T_b}$ is projective, so $H^I\to {T_b}$ is also
  projective, hence finite.  Since $H\to {T_b}$ is flat,
  $\length(H_t)$ is constant.  Furthermore, $I\subseteq H^I$ is open,
  so $H^I_t=I_t$ and hence $\length(H_t)=\length(I_t)$ for a general
  $t\in {T_b}$.  However, we observed above that $\length(I_t)$ is
  also constant, so we must have that $\length(H_t)=\length(I_t)$ for
  all $t\in {T_b}$, and since $I\subseteq H^I$, this means that
  $I=H^I$ and $\iota:I\to {T_b}$ is finite and unramified, hence
  étale.

  In order to prove the global triviality of $Y_I$, consider
  $\Isom_I(Y_I, I\times F_b)$. Recall that taking $\Hilb$ and $\Isom$
  commutes with base change, and so we obtain an isomorphism
  $$
  \Isom_I(Y_I, I\times F_b) \simeq I\times_{T_b}
  \Isom_{T_b}({Y_b}, {T_b}\times F_b) \simeq I \times_{T_b} I.
  $$
  This scheme admits a natural section over ${T_b}$, namely its
  diagonal, which induces an $I$-isomorphism between $Y_I$ and $I
  \times F_b$.
\end{proof}

The preceding Lemma~\ref{lem:dominant} can be used to compare two
families whose associated moduli maps agree. We show that in our setup
any two such families become globally isomorphic after changing base.

\begin{lem}\label{lem:relglue}
  In addition to Assumption~\ref{ass:triplemor}, assume that there
  exists another projective morphism, $Z\to T$, with the following
  property: for any $b\in B$ and any $t\in T_b$, we have $Y_t \simeq
  Z_t \simeq F_b$. Then
  \begin{enumerate-p}
  \item \ilabel{lem:54-1} there exists a surjective morphism $\tau:
    {\wtilde T}\to T$ such that the pull-back families of $Y$ and $Z$
    to $\wtilde T$ are isomorphic as $\wtilde T$-schemes, i.e., we have
    a commutative diagram as follows:
    $$
    \hskip 8em
    \xymatrix{ Y_{\wtilde T} \ar[dr] \ar@/^1.25pc/[rrr]
      \ar@{<->}[rr]_{\wtilde T-\text{isom.}}  & & Z_{\wtilde T}
      \ar[dl] \ar@/^1.25pc/[rrr]|(.17)\hole & Y
      \ar[dr]  && Z \ar[dl] \\
      & \wtilde T \ar[rrr]^\tau & & & T \ar[d]^\rho\\
      & & & & B.}
    $$
  \end{enumerate-p}
  Furthermore, if for all $b\in B$, the group $\Aut(F_b)$ is finite,
  then $\wtilde T$ can be chosen such that the following holds. Let
  $\wtilde T' \subset \wtilde T$ be any irreducible component. Then
  \begin{enumerate-p}\addtocounter{enumi}{1}
  \item\ilabel{lem:54-2} $\tau$ is quasi-finite,
  \item\ilabel{lem:54-3} the image set $\tau (\wtilde T')$ is a union
    of $\rho$-fibers, and
  \item\ilabel{lem:54-4} if $\wtilde T'$ dominates $B$, then there
    exists an open subset $B^\circ \subset (\rho\circ\tau)(\wtilde T')$
    such that $\tau\resto {\wtilde T'}$ is finite and étale over $B^\circ$.
    More precisely, if we set $T^\circ := \rho^{-1}(B^\circ)$ and
    $\wtilde T^\circ := \tau^{-1}(T^\circ) \cap \wtilde T'$, then the
    restriction $\tau\resto {\wtilde T^\circ} : \wtilde T^\circ \to T^\circ$
    is finite and étale.
  \end{enumerate-p}
\end{lem}

\begin{subrem}
  In Lemma~\ref{lem:relglue} we do not claim that $\wtilde T$ is
  irreducible or connected.
\end{subrem}

\begin{proof}[Proof of Lemma~\ref{lem:relglue}]
  Set ${\wtilde T} := \Isom_T(Y,Z)$ and let $\tau:\wtilde T\to T$ be the
  natural morphism.  Again, taking $\Isom$ commutes with base change,
  and we have an isomorphism ${\wtilde T}\times_T{\wtilde T} \simeq
  \Isom_{{\wtilde T}}(Y_{{\wtilde T}}, Z_{{\wtilde T}}).  $ Similarly,
  for all $b\in B$, and for all $t\in T_b$, there is a natural
  one-to-one correspondence between ${{\wtilde T}}_t$ and $\Aut(F_b)$.
  In particular, we obtain that $\tau$ is surjective. As before,
  observe that ${{\wtilde T}}\times_T {{\wtilde T}}$ admits a natural
  section, the diagonal. This shows~\iref{lem:54-1}.
  
  If for all $b\in B$, $\Aut(F_b)$ is finite, then the restriction of
  $\tau$ to any $\rho$-fiber, $\tau_b:{{\wtilde T}}_b\to T_b$ is
  finite étale by Lemma~\ref{lem:dominant}.  This shows
  \iref{lem:54-2} and \iref{lem:54-3}.  Furthermore, it implies that
  if $\wtilde T' \subset \wtilde T$ is a component that dominates $B$,
  neither the ramification locus of $\tau\resto {\wtilde T'}$ nor the
  locus where $\tau\resto {\wtilde T'}$ is not finite dominates $B$.
  In fact, if we let $B^\circ$ denote the open set of $B$ where
  $\#\Aut(F_b)$ is constant, then \iref{lem:54-4} holds for $B^\circ$.
\end{proof}

\subsection{Families where $\rho$ has a section}

Now consider Assumption~\ref{ass:triplemor} in case the morphism
$\rho$ admits a section $\sigma: B \to T$ such that $Z = Y_B \times_B
T$. As a corollary to Lemma~\ref{lem:relglue}, we will show that in
this situation $\wtilde T$ always contains a component $\wtilde T'$
such that the pull-back family $Y_{\wtilde T'}$ comes from $B$.
Better still, the restriction $\tau\resto {\wtilde T'}$ is ``relatively
étale'' in the sense that $\tau\resto {\wtilde T'}$ is étale and that $\rho
\circ \tau\resto {\wtilde T'}$ has connected fibers.

\begin{cor}\label{cor:pulling-back-after-rel-etale}
  Under the conditions of Lemma~\ref{lem:relglue} assume that $\rho$
  admits a section $\sigma: B \to T$, and that $Z = Y_B \times_B T$.
  Then there exists an irreducible component $\wtilde T' \subset
  \wtilde T$ such that
  \begin{enumerate}
  \item \ilabel{cor:56-1} $\wtilde T'$ surjects onto $B$, and
  \item \ilabel{cor:56-2} the restricted morphism $\rho \circ
    \tau\resto {\wtilde T'} : \wtilde T' \to B$ has connected fibers.
  \end{enumerate}
\end{cor}
\begin{proof}
  It is clear from the construction that $Y_B \simeq Z_B$. This
  isomorphism corresponds to a morphism $\wtilde\sigma: B\to
  \Isom_T(Y,Z)=\wtilde T$. Let $\wtilde T' \subset \wtilde T$ be an
  irreducible component that contains the image of $\wtilde \sigma$.
  The existence of a section guarantees that $\rho \circ \tau\resto
  {\wtilde T'} : \wtilde T' \to B$ is surjective and its fibers are
  connected.
\end{proof}

One particular setup where a section is known to exist is when $T$ is
a birationally ruled surface over $B$. The following will become
important later.

\begin{cor}\label{cor:family-push-forward}
  In addition to Assumption~\ref{ass:triplemor}, suppose that $B$ is a
  smooth curve and that the general $\rho$-fiber is isomorphic to
  $\P^1$, $\mathbb A^1$ or $(\mathbb A^1)^*= \mathbb A^1\setminus
  \{0\}$. Then there exist non-empty Zariski open sets $B^\circ
  \subset B$, $T^\circ := \rho^{-1}(B^\circ)$ and a commutative
  diagram
  $$
  \xymatrix{%
    {\wtilde T^\circ} \ar@<-1pt>[r]^{\tau}_{\text{étale}}
    \ar@/_1mm/[rd]_{\text{conn.~fibers}} &
    \text{\vphantom{{$\wtilde T^\circ$}}}
    \ \ {T^\circ} \ar[d]^{\rho} \\
    & \text{\phantom{\ }}{B^\circ} }
  $$
  such that 
  \begin{enumerate-p}
  \item the fibers of $\rho\circ\tau$ are again isomorphic to $\P^1$,
    $\mathbb A^1$ or $(\mathbb A^1)^*$, respectively, and
  \item the pull-back family $Y_{\wtilde T^\circ}$ comes from
    $B^\circ$, i.e., there exists a projective family $Z \to B^\circ$
    and a $\wtilde T^\circ$-isomorphism
    $$
    Y_{\wtilde T^\circ} \simeq Z_{\wtilde T^\circ}.
    $$
  \end{enumerate-p}
\end{cor}
\begin{subrem}\label{srem:family-push-forward}
  If the general $\rho$-fiber is isomorphic to $\P^1$ or $\mathbb
  A^1$, the morphism $\tau$ is necessarily an isomorphism. Shrinking
  $B^\circ$ further, if necessary, $\rho: T^\circ \to B^\circ$ will
  then even be a trivial $\P^1$-- or $\mathbb A^1$--bundle,
  respectively.
\end{subrem}

\begin{proof}
  Shrinking $B$, if necessary, we may assume that all $\rho$-fibers
  are isomorphic to $\P^1$, $\mathbb A^1$ or $(\mathbb A^1)^*$, and
  hence that $T$ is smooth. Then it is always possible to find a
  relative smooth compactification of $T$, i.e. a smooth $B$-variety
  $\overline T \to B$ and a smooth divisor $D \subset T$ such that
  $\overline T \setminus D$ and $T$ are isomorphic $B$-schemes.
  
  By Tsen's theorem, \cite[p.~73]{Shaf94}, there exists a section
  $\sigma: B \to \overline T$.  In fact, there exists a positive
  dimensional family of sections, so that we may assume without loss
  of generality that $\sigma(B)$ is not contained in $D$.
  
  Let $B^\circ \subset B$ be the open subset such that for all $b \in
  B^\circ$, $\overline T_b \simeq \P^1$, $T_b$ is isomorphic to
  $\P^1$, $\mathbb A^1$ or $(\mathbb A^1)^*$, respectively, and
  $\sigma(b) \not \in D$.  Using that any connected finite étale cover
  of $T_b$ is again isomorphic to $T_b$, and shrinking $B^\circ$
  further, Corollary~\ref{cor:pulling-back-after-rel-etale} yields the
  claim.
\end{proof}

\begin{rem}
  Throughout the article we work over the field of complex numbers
  $\bC$, thus we kept that assumption here as well.  However, we would
  like to note that the results of this section work over an arbitrary
  algebraically closed base field $k$.
\end{rem}

\part{PROOF OF THEOREM~\ref*{thm:mainresult}}

\section{Setup and Notation}

The cases $\kappa(S^\circ)= -\infty$, $0$ and $1$ are considered
separately in Sections~\ref{sec:kinfty}--\ref{sec:k1} below.

The following setup and notation will be used throughout the rest of
the article: As in Theorem~\ref{thm:mainresult}, we fix a smooth
compactification $S^\circ \subset S$ such that $D := S\setminus
S^\circ$ is a divisor with simple normal crossings. The log minimal
model program then yields a birational morphism $\lambda : S \to
S_\lambda$, with the following properties.
\begin{enumerate}
\item The surface $S_\lambda$ is normal and $\mathbb Q$-factorial.
\item If $D_\lambda$ is the cycle-theoretic image of $D$, then
  $(S_\lambda, D_\lambda)$ is a reduced {dlt} pair.
\item By Lemma~\ref{ex:dltsurfisfd} and
  Theorem~\ref{thm:logformextension}, the extension theorem holds for
  $(S_\lambda, D_\lambda)$.
\end{enumerate}
Again, recall from \cite[Thm.~1.4]{VZ02} that there exists a
Viehweg-Zuo sheaf $\sA \subset \Sym^n \Omega^1_S(\log D)$ of positive
Kodaira-Iitaka dimension $\kappa(\sA) \geq \Var(f^\circ) > 0$.  By
Lemma~\ref{lem:pushdownA}, there exists a Viehweg-Zuo sheaf
$\sA_\lambda \subset \Sym^{[n]} \Omega^1_{S_\lambda}(\log D_\lambda)$
of Kodaira-Iitaka dimension $\kappa(\sA_\lambda) \geq \kappa(\sA)$.

\section{Proof in case \texorpdfstring{$\kappa(S^\circ)=-\infty$}{Kodaira dimension minus infinity}}
\label{sec:kinfty}

In this case, the log canonical bundle $K_S +D$ has negative
Kodaira-Iitaka dimension, and $(S_\lambda, D_\lambda)$ is a pair that
either has the structure of a Mori-Fano fiber space or is a log-Fano
pair with Picard number $\rho(S_\lambda)=1$. However, since the
extension theorem holds, Theorem~\ref{thm:b2thm} rules out the case
that $\rho(S_\lambda)=1$. The pair $(S_\lambda, D_\lambda)$ thus
always admits a fibration, independently of the choices made in its
construction.  In particular, there exists a smooth curve $C$ and a
fibration $\pi_\lambda: S_\lambda \to C$ with connected fibers, such
that $-(K_{S_\lambda} + D_\lambda)$ intersects the general fiber
positively.

Setting $\pi := \pi_\lambda \circ \lambda$, the general fiber $F$ of
$\pi$ is then a rational curve that intersects the boundary in one
point, if at all.  In particular, the restriction of the family
$f^\circ$ to $F \cap S^\circ$ is necessarily isotrivial by
\cite{Kovacs00a}. The detailed description of the moduli map in case
$\kappa(S^\circ)=-\infty$ then follows from
Corollary~\ref{cor:family-push-forward} and
Remark~\ref{srem:family-push-forward}. \qed

\section{Proof that \texorpdfstring{$\kappa(S^\circ) \not = 0$}{Kodaira dimension is not zero}}
\label{sec:k0}

\subsection{Setup}

To prove Theorem~\ref{thm:mainresult} in this case, we argue by
contradiction and assume that $\kappa(S^\circ) = 0$. Let $(\wtilde
S_\lambda, \wtilde D_\lambda)$ be the index-one cover of a log-minimal
model, as in Proposition~\ref{prop:index-cover}.  The main properties
of $(\wtilde S_\lambda, \wtilde D_\lambda)$ are summarized as follows.
\CounterStep
\begin{enumerate-p}
\item The pair $(\wtilde S_\lambda, \wtilde D_\lambda)$ is $\mathbb
  Q$-factorial and {dlt} (Proposition~\ref{prop:index-cover} and
  Corollary~\ref{cor:singofcoverdim2}).
\item\ilabel{il:812} There exists a Viehweg-Zuo sheaf $\wtilde
  \sA_\lambda \subset \Sym^{[n]} \Omega^1_{\wtilde S_\lambda}(\log
  \wtilde D_\lambda)$ of positive Kodaira-Iitaka dimension
  (Lemma~\ref{lem:VZontilde}).
\item\ilabel{il:813} If $\wtilde S_{\lambda,\reg} \subset \wtilde
  S_\lambda$ is the maximal smooth open subset, then the restriction
  $\wtilde \sA_\lambda\resto{\wtilde S_{\lambda,\reg}}$ is invertible
  (\cite[1.1.15 on p.~154]{OSS}).
\item $\wtilde S_\lambda$ is uniruled, and the boundary $\wtilde
  D_\lambda$ is not empty (Proposition~\ref{prop:unirulednessofbaseK0}
  and Corollary~\ref{cor:boundarynonempty}).
\item The divisor $K_{\wtilde S_\lambda} + \wtilde D_\lambda$ is
  Cartier and trivial (Proposition~\ref{prop:index-cover}).
\item\ilabel{il:815} If $p \in \wtilde D_\lambda$ is any point, then
  $(\wtilde S_\lambda, \wtilde D_\lambda)$ is log-smooth at $p$. In
  particular, $\Omega^1_{\wtilde S_\lambda}(\log \wtilde D_\lambda)$
  is locally free along $\wtilde D_\lambda$
  (Corollary~\ref{cor:singofcoverdim2}).
\end{enumerate-p}

\subsection{Outline of the proof}

As a first step in the proof of Theorem~\ref{thm:mainresult}, we aim
to apply Theorem~\ref{thm:b2thm}, in order to show that $\wtilde
S_\lambda$ is fibered over a curve, with rational fibers that
intersect the boundary twice. Since Theorem~\ref{thm:b2thm} works best
in the case $\kappa = -\infty$ we need to decrease the boundary
coefficients slightly and perform extra contractions before
Theorem~\ref{thm:b2thm} can be applied to prove the existence of a
fibration.

The fiber space structure of $\wtilde S_\lambda$ is then used to
analyze the restriction of the Viehweg-Zuo sheaf $\wtilde \sA_\lambda$
to a suitable boundary component $\wtilde D'_\lambda \subset \wtilde
D_\lambda$.  Even though there is no smooth family over $\wtilde
D'_\lambda$, it will turn out that the restriction $\wtilde
\sA_\lambda\resto{\wtilde D'_\lambda}$ can be interpreted as a
Viehweg-Zuo sheaf on $\wtilde D'_\lambda$, which again has positive
Kodaira-Iitaka dimension. This leads to contradiction and thus
finishes the proof.

\subsection{Minimal models of $\boldsymbol{(\wtilde S_\lambda, \wtilde D_\lambda)}$}

Since $\wtilde D_\lambda$ is not empty and $K_{\wtilde S_\lambda}
\equiv -\wtilde D_\lambda$, it follows that for any rational number $0
< \varepsilon < 1$,
\begin{equation*}
  \kappa\bigl({\wtilde S_\lambda},(1-\varepsilon)\wtilde D_\lambda \bigr) =
  \kappa\bigl(K_{\wtilde S_\lambda}+(1-\varepsilon)\wtilde D_\lambda \bigr) =
  \kappa\bigl(\varepsilon \cdot K_{\wtilde S_\lambda} \bigr) =
  \kappa\bigl({\wtilde S_\lambda} \bigr) = -\infty.
\end{equation*}
Choose a rational number $0 < \varepsilon < 1$ and perform a minimal
model program for the pair $\bigl(\wtilde S_\lambda,
(1-\varepsilon)\wtilde D_\lambda\bigr)$. This will produce a
birational morphism $\mu: \wtilde S_\lambda \to S_\mu$.  Let $D_\mu$
be the cycle-theoretic image of $\wtilde D_\lambda$.  Since
$\bigl(\wtilde S_\lambda, (1-\varepsilon) \wtilde D_\lambda \bigr)$
has {dlt} singularities, the pair $\bigl( S_\mu, (1-\varepsilon)
\wtilde D_\mu \bigr)$ will also be {dlt}, in fact, it will be {klt}.

\begin{rem}\label{rem:Mori-Fano}
  Since $\kappa({\wtilde S_\lambda}, (1-\varepsilon) \wtilde
  D_\lambda) =-\infty$, either $\rho(S_\mu)>1$ and the pair $\bigl(
  S_\mu, (1-\varepsilon) D_\mu\bigr)$ is a Mori-Fano fiber space, or
  $\rho(S_\mu)=1$.
\end{rem}

\begin{rem}\label{rem:smuklt}
  It follows from the equation $K_{\wtilde S_\lambda} \equiv -\wtilde
  D_\lambda$ that for any $0 < \varepsilon', \varepsilon'' < 1$, the
  divisors $K_{\wtilde S_\lambda} + (1-\varepsilon') \wtilde
  D_\lambda$ and $K_{\wtilde S_\lambda} + (1-\varepsilon'') \wtilde
  D_\lambda$ are multiples of one another.  In particular, the
  birational morphism $\mu$ is a minimal model program for the pair
  $\bigl(\wtilde S_\lambda, (1-\varepsilon) \wtilde D_\lambda \bigr)$,
  independently of the chosen $0 < \varepsilon < 1$. It follows that
  $\bigl(S_\mu, (1-\varepsilon)D_\mu \bigr)$ has {dlt} singularities
  for all $\varepsilon$. In particular, it follows directly from the
  definition of discrepancy
  \PreprintAndPublication{\cite[2.26]{KM98}}{} that the reduced pair
  $(S_\mu, D_\mu)$ is {dlc} in the sense of Definition~\ref{def:dlc}.
\end{rem}

\subsection{The fiber space structure of $\boldsymbol{\wtilde S_\lambda}$}

We apply Theorem~\ref{thm:b2thm} in order to show that $S_\mu$ is a
fiber space.

\begin{prop}\label{prop:mfs}
  One has that $\rho(S_\mu) > 1$.  In particular, $S_\mu$ has the
  structure of a non-trivial Mori-Fano fiber space.
\end{prop}
\begin{proof}
  If $-(K_{S_\mu}+D_\mu)$ is not ample, then $\rho(S_\mu) > 1$ and the
  statement follows from Remark~\ref{rem:Mori-Fano}.  If
  $-(K_{S_\mu}+D_\mu)$ is ample, then Remark~\ref{rem:smuklt} and
  Example~\ref{ex:dltsurfisfd} imply that $(S_\mu, D_\mu)$ is finitely
  dominated by smooth analytic pairs.  Then by
  Theorem~\ref{thm:logformextension}, the extension theorem holds for
  $(S_\mu, D_\mu)$.  According to Lemma~\ref{lem:pushdownA} there
  exists a Viehweg-Zuo subsheaf $\sA_\mu \subset \Sym^{[n]}
  \Omega^1_{S_\mu}(\log D_\mu)$ of positive Kodaira-Iitaka dimension
  and then Theorem~\ref{thm:b2thm} and Remark~\ref{rem:Mori-Fano}
  imply the desired statement.
\end{proof}

\begin{cor}\label{cor:descofF}
  There exists a morphism $\pi : \wtilde S_\lambda \to C$ to a smooth
  curve, and an open set $C^\circ \subset C$ such that for any $c \in
  C^\circ$, the associated fiber $F_c := \pi^{-1}(c)$ is a smooth
  rational curve which is entirely contained in the log-smooth locus
  $(\wtilde S_\lambda, \wtilde D_\lambda)_{\reg}$ and which intersects
  the boundary $\wtilde D_\lambda$ transversally in exactly two points.
  In particular, the sheaf $\Omega^1_{\wtilde S_\lambda}(\log \wtilde
  D_\lambda)\resto{F_c}$ is trivial.
\end{cor}
\begin{proof}
  The existence of $\pi$ and the rationality of the general fiber
  follows from Proposition~\ref{prop:mfs}. The number of intersection
  points follows from $K_{\wtilde S_\lambda} + \wtilde D_\lambda
  \equiv 0$.\PreprintAndPublication{ The triviality of
    $\Omega^1_{\wtilde S_\lambda}(\log \wtilde D_\lambda)\resto{F_c}$
    follows from standard sequences, see \cite[2.14]{KK05} and
    \eqref{eq:restoflogomega2} below.}{}
\end{proof}

\begin{cor}\label{K0pf:cor-3}
  If $c \in C^\circ$ is a general point, then the restriction $\wtilde
  \sA_\lambda\resto{F_c}$ is trivial.
\end{cor}
\begin{proof}
  Since $F_c$ is a general fiber, $\wtilde
  \sA_\lambda^{[r]}\resto{F_c}$ is an invertible sheaf for any
  $r\in\bZ$ by \iref{il:813}. In particular,
  $\wt\sA_\lambda^{[r]}\resto{F_c} \simeq
  \bigl(\wt\sA_\lambda\resto{F_c}\bigr)^{\otimes r}$. Fix an $r \in
  \mathbb N$ such that $h^0\bigl(\wtilde S_\lambda,\, \wtilde
  \sA_\lambda^{[r]}\bigr) > 0$.  Then there exists a non-trivial and
  hence injective morphism
  $$
  \wtilde \sA_\lambda^{[r]}\resto{F_c} \to \bigl(\Sym^{[r\cdot n]}
  \Omega^1_{\wtilde S_\lambda}(\log \wtilde D_\lambda)
  \bigr)\resto{F_c} \simeq \Sym^{r \cdot n} \bigl( \Omega^1_{\wtilde
    S_\lambda}(\log \wtilde D_\lambda)\resto{F_c} \bigr).
  $$
  The triviality of the sheaf $\Omega^1_{\wtilde S_\lambda}(\log
  \wtilde D_\lambda)\resto{F_c}$ implies that $\deg \left(\wtilde
    \sA_\lambda^{[r]}\resto{F_c}\right)\leq 0$. Since $F_c$ passes
  through a general point of $\wtilde D_\lambda$, a general section of
  the sheaf $\wtilde \sA_\lambda^{[r]}$ does not vanish along all of
  $F_c$.  Therefore $\wtilde \sA_\lambda^{[r]}\resto{F_c}$ is a line
  bundle of non-positive degree that has a global section.
  Consequently it is trivial. Since $F_c\simeq \bP^1$, this implies
  the statement.
\end{proof}

\subsection{Non-triviality of $\boldsymbol{\wtilde {\scr{A}}_\lambda\resto{D}}$} 

Now consider a section $\sigma \in H^0\bigl(\wtilde S_\lambda,\, \wtilde
\sA_\lambda^{[r]}\bigr)$, let $F_c$ be a general $\pi$-fiber and $y \in
F_c$  a general point of $F_c$. The triviality of $\wtilde
\sA_\lambda^{[r]}$ on $F_c$ can now be used to compare the value of
$\sigma$ at a $y$ with its value at a point where $F_c$ hits the boundary
$\wtilde D_\lambda$. It will follow that $\sigma$ is completely determined
by the values it takes on the boundary.

\begin{lem}\label{lem:restofposkod}
  There exists an irreducible component $\wtilde D_\lambda' \subset
  \wtilde D_\lambda$ such that for any $r \in \mathbb N$, the natural
  restriction morphism
  $$
  H^0\bigl(\wtilde S_\lambda,\, \wtilde \sA_\lambda^{[r]}\bigr) \to
  H^0\bigl(\wtilde D'_\lambda,\, \wtilde \sA_\lambda^{[r]}\resto{\wtilde
    D'_\lambda}\bigr)
  $$
  is injective. In particular, the restriction $\wtilde
  \sA_\lambda\resto{\wtilde D'_\lambda}$ is a non-trivial invertible
  sheaf and its Kodaira-Iitaka dimension equals the Kodaira-Iitaka
  dimension of $\wtilde \sA_\lambda$, i.e., $\kappa(\wtilde
  \sA_\lambda) = \kappa\bigl(\wtilde \sA_\lambda\resto{\wtilde
    D'_\lambda}\bigr)$.
\end{lem}

\begin{proof}
  Corollary~\ref{cor:descofF} implies that there exists a component
  $\wtilde D'_\lambda \subset \wtilde D_\lambda$ that is hit by all
  the curves $(F_c)_{c \in C^\circ}$. Now let $r$ be any given number.
  If $h^0\bigl(\wtilde S_\lambda,\, \wtilde \sA_\lambda^{[r]}\bigr) =
  0$, there is nothing to show.  Otherwise, using the notation of
  Corollary~\ref{cor:descofF}, set 
  $$
  \wtilde D'_{\lambda,\circ} := \wtilde D'_{\lambda} \cap
  \pi^{-1}(C^\circ).
  $$
  Since a section in the trivial bundle is determined by its value
  at any given point, a section $\sigma \in H^0\bigl(\wtilde
  S_\lambda,\, \wtilde \sA_\lambda^{[r]}\bigr)$ is uniquely determined
  by its restriction to $\wtilde D'_{\lambda,\circ} \subset \wtilde
  D'_\lambda$ by Corollary~\ref{K0pf:cor-3}.  Finally, note that
  $\wtilde \sA_\lambda\resto{\wtilde D'_\lambda}$ is invertible by
  \iref{il:813} and \iref{il:815}.  Thus the claim is shown.
\end{proof}

\begin{rem}\label{rem:restiota}
  Let $\iota : \wtilde \sA_\lambda \to \Sym^{[n]} \Omega^1_{\wtilde
    S_\lambda}(\log \wtilde D_\lambda)$ denote the injection of our
  Viehweg-Zuo sheaf into the sheaf of pluri-log differentials. Then
  its restriction $\iota\smallresto{\wtilde D'_\lambda}$ is injective.
\end{rem}

\subsection{Existence of pluri-forms on $\boldsymbol{\wtilde D'_\lambda}$} 

As a last ingredient in the proof, we show that sections in tensor
powers of the invertible sheaf $\wtilde \sA_\lambda\resto{\wtilde
  D'_\lambda}$ can again be interpreted as pluri-forms on the
boundary.

\begin{lem}\label{lem:pluriformsondprime}
  Let $\wtilde D''_\lambda = (\wtilde D_\lambda - \wtilde
  D'_\lambda)\resto{\wtilde D'_\lambda}$. Then there exists a number
  $m \leq n$ and an injective sheaf morphism $\wtilde
  \sA_\lambda\resto{\wtilde D'_\lambda} \to \Sym^m \Omega^1_{\wtilde
    D'_\lambda}(\log \wtilde D''_\lambda)$.
\end{lem}
\begin{proof}
  Since $\wt \sA_\lambda\resto{\wtilde D'_\lambda}$ is invertible by
  \iref{il:813} and \iref{il:815}, it is enough to show that there
  exists a non-zero morphism $\wtilde \sA_\lambda\resto{\wtilde
    D'_\lambda} \to \Sym^m \Omega^1_{\wtilde D'_\lambda}(\log \wtilde
  D''_\lambda)$, for some $m\in \bN$.  We will use the following
  sequence that relates restrictions of log-forms with log-forms on
  the restriction---the sequence is discussed in \cite[2.13]{KK05}.
  \begin{equation}\label{eq:restoflogomega2}
    \xymatrix{
      0 \ar[r] & \Omega^1_{\wtilde D'_\lambda}(\log \wtilde
      D''_\lambda) \ar[r]^(.45){\alpha} &  
      \Omega^1_{\wtilde S_\lambda}(\log \wtilde
      D_\lambda)\resto{\wtilde D'_\lambda} \ar[r]^(.65){\beta} &
      \O_{\wtilde D'_\lambda} \ar[r] & 0. 
    }
  \end{equation}
  Along with this sequence comes the standard filtration of the
  symmetric product,
  $$
  \Sym^n \Omega^1_{\wtilde S_\lambda}(\log \wtilde
  D_\lambda)\resto{\wtilde D'_\lambda} = \sF^0 \supseteq \sF^1
  \supseteq \cdots \supseteq \sF^n \supseteq \sF^{n+1} = 0,
  $$
  with quotients
  \begin{equation}\label{eq:filtration}
    \xymatrix{
      0 \ar[r] & \sF^{p+1} \ar[r]^{\alpha_p} & \sF^p \ar[r]^(.27){\beta_p}
      & \Sym^p \Omega^1_{\wtilde D'_\lambda}(\log \wtilde D''_\lambda)
      \ar[r] & 0. 
    }
  \end{equation}
  \PreprintAndPublication{See \cite[ex.~II.5.16]{Ha77} for details.}{}
  As in Remark~\ref{rem:restiota}, let $\iota$ be the injection of the
  Viehweg-Zuo sheaf $\wt\sA_\lambda$ into the sheaf of pluri-log
  differentials $\Sym^{[n]} \Omega^1_{\wtilde S_\lambda}(\log \wtilde
  D_\lambda)$.  Recall from Lemma~\ref{lem:restofposkod} that $\wtilde
  \sA_\lambda\resto{\wtilde D'_\lambda}$ has positive Kodaira-Iitaka
  dimension and from Remark~\ref{rem:restiota} that it embeds into
  $\Sym^{[n]} \Omega^1_{\wtilde S_\lambda}(\log \wtilde
  D_\lambda)\resto {\wtilde D'_\lambda}\simeq \Sym^{n}
  \Omega^1_{\wtilde S_\lambda}(\log \wtilde D_\lambda)\resto {\wtilde
    D'_\lambda}$.
  
  First consider the sequence in \eqref{eq:filtration} for $p = 0$.
  Since $\Sym^0 \Omega^1_{\wtilde D'_\lambda}(\log \wtilde
  D''_\lambda) = \O_{\wtilde D'_\lambda}$, and since any morphism from
  an invertible sheaf of positive Kodaira-Iitaka dimension to the
  structure sheaf is necessarily zero, the composition $\beta_0 \circ
  \iota\smallresto{\wtilde D'_\lambda}$ is zero, and the restriction
  $\iota\smallresto{\wtilde D'_\lambda}$ factors via an injection
  $\iota_1: \wtilde \sA\resto{\wtilde D'} \to \sF^1$.
  
  Next consider \eqref{eq:filtration} for $p = 1$. If $\beta_1 \circ
  \iota_1$ is non-zero, the proof is finished.  Otherwise, $\iota_1$
  factors via an injection $\iota_2: \wtilde \sA_\lambda\resto{\wtilde
    D'_\lambda} \to \sF^2$, and we consider \eqref{eq:filtration} for
  $p = 2$, etc. This process must stop after no more than $n$ steps.
  Thus the claim is shown.
\end{proof}

\subsection{End of the proof}

Using the notation introduced in Lemma~\ref{lem:pluriformsondprime},
the adjunction formula shows that $K_{\wtilde D'_\lambda} + \wtilde
D''_\lambda = (K_{\wtilde S_\lambda} + \wtilde
D_\lambda)\resto{\wtilde D'_\lambda} \equiv 0$. In particular, $\deg
\Omega^1_{\wtilde D'_\lambda}(\log \wtilde D''_\lambda) = 0$. On the
other hand, Lemma~\ref{lem:pluriformsondprime} asserts the existence
of an injective morphism of sheaves $\wtilde \sA_\lambda\resto{\wtilde
  D'_\lambda} \to \Sym^m \Omega^1_{\wtilde D'_\lambda}(\log \wtilde
D''_\lambda)$. By Lemma~\ref{lem:restofposkod}, $\wtilde
\sA_\lambda\resto {\wtilde D'_\lambda}$ has positive Kodaira-Iitaka
dimension $\kappa \bigl(\wtilde \sA_\lambda\resto{\wtilde
  D'_\lambda}\bigr) = \kappa(\wtilde \sA_\lambda) > 0$.  This is
clearly absurd, and the proof of Theorem~\ref{thm:mainresult} is thus
finished in the case $\kappa(S^\circ)=0$.  \qed

\section{Proof in case \texorpdfstring{$\kappa(S^\circ)=1$}{Kodaira dimension one}}
\label{sec:k1}

In this case the statements of Theorem~\ref{thm:mainresult} follow
from the results of Section~\ref{sec:gluearama} when one applies the
logarithmic minimal model program. The following proposition
summarizes the standard description of surfaces with logarithmic
Kodaira dimension 1.

\begin{prop}\label{prop:cstarortorus}
  If $\kappa({S^\circ})=1$, then there exists a smooth curve $C$ and a
  fibration $\pi: S \to C$ with connected fibers, such that $K_S + D$
  is trivial on the general fiber. In particular, one of the following
  holds:
  \begin{enumerate}
  \item\ilabel{il:torus} The general fiber is an elliptic curve and no
    component of $D$ dominates $C$, or
  \item\ilabel{il:cstar} The general fiber is isomorphic to $\P^1$ and
    $D$ intersects the general fiber in exactly two points.
  \end{enumerate}
\end{prop}
\begin{proof}
  The logarithmic abundance theorem in dimension 2, see
  e.g.~\cite[3.3]{KM98}, asserts that for $n \gg 0$ the linear system
  $|n(K_{S_{\lambda}}+D_{\lambda})|$ yields a morphism to a curve
  $\pi_\lambda : S_\lambda \to C$, such that $K_{S_{\lambda}} +
  D_{\lambda}$ is trivial on the general fiber $F_\lambda$ of
  $\pi_\lambda$. Likewise, if $\pi := \pi_\lambda \circ \phi$ and $F
  \subset S$ is a general fiber of $\pi$, then $K_S + D$ is trivial on
  $F$. Statements~\iref{il:torus} and \iref{il:cstar} describe the
  only two ways this can happen.
\end{proof}

To finish the proof of Theorem~\ref{thm:mainresult}, consider the
morphism $\pi:S\to C$ provided by Proposition~\ref{prop:cstarortorus}.
Let $V\subseteq C$ be the locus over which $\pi$ is smooth and either
$D\cap \pi^{-1}(V)=\emptyset$ or $\pi\resto D$ is \'etale.  Consider
the restriction of $\pi$ to $U := \pi^{-1}(V)\cap S^\circ$. By
Proposition~\ref{prop:cstarortorus}, the general fiber of $\pi\resto
U$ is either an elliptic curve, or it is isomorphic to $\C^*$.  In
both cases, it follows from \cite{Kovacs96e} and \cite{Kovacs00a} that
$f$ is isotrivial on the fibers of $\pi: U\to V$. The factorization of
the moduli map follows.
  
It remains to give the detailed description of the moduli map. If the
general fibers of $\pi$ are isomorphic to $\mathbb C^*$,
Corollary~\ref{cor:family-push-forward} yields the claim. Otherwise,
take an irreducible multisection $\widehat V\subset S$, restrict $V$
further if necessary so $\widehat V$ is \'etale over $V$ and take a
base change to $\widehat V$. We end up with a section $\sigma:
\widehat V\to \widehat U := U \times_V \widehat V$.  Finally, set
$\widehat X := X \times_U \widehat U$, and $Z := \widehat V
\times_\sigma \widehat X$. Shrinking $V$ further, if necessary, an
application of Lemma~\ref{lem:relglue} completes the proof of
Theorem~\ref{thm:mainresult}. \qed

\providecommand{\bysame}{\leavevmode\hbox to3em{\hrulefill}\thinspace}
\providecommand{\MR}{\relax\ifhmode\unskip\space\fi MR}
\providecommand{\MRhref}[2]{%
  \href{http://www.ams.org/mathscinet-getitem?mr=#1}{#2}
}
\providecommand{\href}[2]{#2}

\end{document}